\documentclass[a4paper,10pt,reqno]{amsart}

\usepackage{amsmath,amsthm,amssymb,amsfonts, xcolor}
\usepackage{bbm}
\usepackage{graphicx}
\usepackage{verbatim}
\usepackage[colorlinks=true,urlcolor=blue,citecolor=red,linkcolor=blue,linktocpage,pdfpagelabels,bookmarksnumbered,bookmarksopen]{hyperref}
\usepackage[utf8]{inputenc} 
\usepackage{mathrsfs}

\newtheorem{theorem}{Theorem}[section]

\newtheorem{lemma}[theorem]{Lemma}
\newtheorem{proposition}[theorem]{Proposition}
\theoremstyle{definition}
\newtheorem{remark}[theorem]{Remark}
\newtheorem{definition}[theorem]{Definition}

\theoremstyle{plain}
\newtheorem*{Alex-pb}{Alexandrov's Problem}
\newtheorem*{NLK-pb}{(NLK) Problem}
\newcommand{\R}{\mathbb{R}}
\newcommand{\N}{\mathbb{N}}
\newcommand{\Hbb}{\mathbb{H}}
\newcommand{\Sbb}{\mathbb{S}}
\newcommand{\hyp}{\Hbb^{m+1}}
\newcommand{\Sm}{\mathbb{S}^m}
\newcommand{\dS}{d\mathbb{S}^{m+1}}
\newcommand{\cA}{\mathcal{A}}
\newcommand{\cB}{\mathcal{B}}
\newcommand{\cC}{\mathcal{C}}
\newcommand{\cF}{\mathcal{F}}
\newcommand{\cK}{\mathcal{K}}
\newcommand{\cL}{\mathcal{L}}
\newcommand{\cO}{\mathcal{O}}
\newcommand{\sC}{\mathscr{C}}
\newcommand{\sL}{\mathscr{L}}
\newcommand{\sF}{\mathscr{F}}
\newcommand{\dom}{\partial\Omega}
\newcommand{\hpi}{\frac{\pi}{2}}
\newcommand{\ind}{\mathbbm{1}}
\newcommand{\spt}{\mathrm{Spt}}
\newcommand{\can}{\mathrm{can}}
\newcommand{\conv}{\mathrm{conv}}
\newcommand{\ex}{\mathrm{e}}
\newcommand{\dl}{d\ell}

\numberwithin{equation}{section}

\DeclareMathOperator{\argth}{argth}

\renewcommand{\P}{\mathcal{P}}
\begin{document}

%
\title[Prescribing the Gauss curvature of convex bodies in $\hyp$]{Prescribing the Gauss curvature of convex bodies in hyperbolic space}

\author{J\'er\^ome Bertrand and Philippe Castillon}

\subjclass[2010]{}

\begin{abstract}
	The Gauss curvature measure of a pointed Euclidean convex body is a measure on the unit sphere which extends the notion of Gauss curvature to non-smooth bodies. Alexandrov's problem consists in finding a convex body with given curvature measure. In Euclidean space, A.D.~Alexandrov gave a necessary and sufficient condition on the measure for this problem to have a solution.

	In this paper, we address Alexandrov's problem for convex bodies in the hyperbolic space $\hyp$. After defining the Gauss curvature measure of an arbitrary hyperbolic convex body, we completely solve Alexandrov's problem in this setting. Contrary to the Euclidean case, we also prove the uniqueness of such a convex body. The methods for proving existence and uniqueness of the solution to this problem are both new.

\end{abstract}

\maketitle

%
\begin{center}
	\begin{minipage}{10cm}
		\small
		\tableofcontents
	\end{minipage}
\end{center}

%
%
\section*{Introduction}
%
%
%

\subsubsection*{Alexandrov's problem in Euclidean space}
The geometry of convex bodies in Euclidean space is described by a finite family of geometric measures on the unit sphere. These measures appear when considering the volume $|\Omega_{\varepsilon}|$ of the $\varepsilon$-neighbor\-hood of a convex body $\Omega$. A non-smooth version of Steiner's formula asserts that the Taylor expansion of $|\Omega_{\varepsilon}|$ is a polynomial, and the coefficients are measures supported on the boundary of the body (and normal vectors). These measures are then gathered into two classes: the area measures and the curvature measures. A natural question is to find necessary and sufficient conditions for a measure on $\Sm \sim \dom$ to be one of these geometric measures; this problem already appears in the work of H.~Minkowski and A.D. Alexandrov \cite{Alexandrov_book}. A standard reference for area and curvature measures is R. Schneider's book \cite[\S4 and \S8]{Schneider}; we also refer to \cite{Guan-Lin-Ma} and references therein for results on the curvature measures.

Among the curvature measures is the \emph{Gauss curvature measure} whose definition depends on a fixed point $o$ in the interior of the convex body. When the boundary of the convex body is $C^2$ and strictly convex, this measure is simply  the pull-back of $Kdv_{\dom}$ (where $K$ is the Gaussian curvature of $\dom$ and $dv_{\dom}$ is its Riemannian measure) on the unit sphere about $o$ using the corresponding radial homeomorphism $P: \Sm \longrightarrow \dom$. For an arbitrary convex body, the Gauss curvature measure is defined as the push-forward of the uniform measure $\sigma$ on the sphere using the inverse of the Gauss map $G: \dom \longrightarrow \Sm$ (see \cite[Chapter 1 \S5]{Bakelman} for more on this point); in particular, the Gauss curvature measure and the uniform measure have \emph{the same total mass}. For a convex polytope, the Gauss curvature measure turns out to be the sum of weighted Dirac masses over the set of unit vectors pointing to the vertices, the weights being the exterior solid angles.

The prescription problem for the Gauss curvature measure is known as ``Alexandrov's problem'', as it was first studied and solved by A.D. Alexandrov \cite[\S 9.1]{Alexandrov_book}. Since this problem is our main concern in this paper, in what follows we will omit the word "Gauss" and simply write "curvature measure". A.D. Alexandrov found a necessary and sufficient condition, referred to as \emph{Alexandrov's condition} in this paper, for a measure to be the curvature measure of a Euclidean convex body. Since then, many other proofs of his result were found; almost all of them are based on the same strategy. We believe it is important to briefly summarize this strategy in order to highlight the difficulties you face when trying to solve Alexandrov's problem for hyperbolic convex bodies. The scheme of proof goes as follows. Starting from a measure $\mu$ satisfying Alexandrov's condition, first (weakly) approximate $\mu$ by a sequence of nicer measures $(\mu_k)_{k\in\N}$ (either finitely supported or with smooth densities w.r.t. $\sigma$) also satisfying Alexandrov's condition. Then, solve the problem for this class of nicer measures (either by Alexandrov's topological approach in the case of finitely supported measures or by PDEs methods for smooth ones), set $\Omega_k$ a solution to the problem for $\mu_k$. Next, up to extracting a subsequence, prove that the sequence of $(\Omega_k)_{k \in \N}$ converges to $\Omega_{\infty}$ with respect to Hausdorff distance, then prove that the Gauss curvature measure $\mu_k$ weakly converges to that of $\Omega_{\infty}$; finally, conclude using that, by construction, $\mu_k \rightharpoonup \mu$. In the penultimate step, it is crucial that the  curvature measure is invariant by the dilations fixing $o$, so that we can force the $\Omega_k$'s to lie in a fixed compact set and thus extract a converging subsequence. As we shall see, this invariance by dilations is no longer true in the hyperbolic setting.

We also emphasize that Alexandrov's problem admits a unique solution up to dilation with respect to $o$. This part of the proof, both for polytopes or arbitrary convex bodies, is often delicate \cite{Alexandrov}. Recently, a variational approach to Alexandrov's problem has been implemented by V. Oliker \cite{Oliker-2}, where his proof builds on the clever fact that Alexandrov's problem can be rephrased in terms of the so-called \emph{Kantorovich's dual problem}, a classical tool in the theory of optimal mass transport. Note however that the cost function involved in this version of Kantorovich's problem is non-standard and present difficulties, mainly because  it is not real-valued. Building on V. Oliker's remark and the fact that the total masses of the Gauss curvature and uniform measures are the same, the first author found a purely optimal mass transport approach to solve Alexandrov's problem \cite{Bertrand-2}.

%
\subsubsection*{The hyperbolic setting}
The aim of this paper is to state and solve a hyperbolic version of Alexandrov's problem. Prior to this work, the problem of prescribing the Gaus curvature was mainly studied for smooth convex bodies by PDEs methods. Indeed, considering the boundary as a radial graph $P: \Sm \longrightarrow \dom$, the prescription problem can be rephrased as a PDE of Monge-Amp\`ere type \cite{Oliker-1,Gerhardt}. In addition to that case, A.D. Alexandrov claims without proof in his book \cite[\S 9.3.2]{Alexandrov_book} that the prescription problem can be solved for hyperbolic convex polyhedra in $\Hbb^3$ with the same proof as in the Euclidean setting. Last, a generalization of Alexandrov's problem, considered as a result on prescribed embeddings of the sphere into Euclidean space, is proved in \cite{Bertrand-1}; it concerns hyperbolic orbifolds. 

First, we point out that while the definition of curvature measure for smooth or polyhedral hyperbolic convex bodies can be directly derived from the Euclidean one, the non-trivial holonomy in $\hyp$ makes the definition for arbitrary convex bodies non-trivial; namely, the pull-back of (normal) vectors from the boundary $\dom$ to the fixed point $o$ depends on the chosen path. Another significant difference with the Euclidean case is the behavior of the curvature measure with respect to dilations which is rather intricate, we refer to Remark \ref{curv_non_homo} for more on this point.

Our approach to circumvent the problem of defining the curvature measure is twofold. First, we replace the unit sphere by the de Sitter space $\dS$ in the definition of the Gauss map $G: \dom \longrightarrow \dS$. Second, we use a Lorentzian counterpart of the classical polar transform of Euclidean convex body. This Lorentzian polar transform provides us with a polar convex body $\Omega^*\subset\dS$ whose boundary can be equipped with an area measure. Once these steps are proved, we end up with a natural curvature measure which coincides, in the polytope and smooth cases, with the previous ones.  While this approach based on duality is part of folklore in the polytope case, we are not aware of such a generalization to arbitrary convex bodies elsewhere in the literature. Another method to define a hyperbolic curvature measure, based on Steiner's formula, is proposed in \cite{Kohlmann}; see Remark \ref{rem-curvature_measures_coincide} for more, including a proof that both approaches lead to the same measure. 

The main result of the paper is
\begin{theorem}\label{thm-main}
	Let $\sigma$ be the uniform measure on $\Sm$. A finite measure $\mu$ on $\Sm$ is the curvature measure of some convex body of $\hyp$ if and only if the following conditions are satisfied:
	\begin{enumerate}
		\item $\mu(\Sm)>\sigma(\Sm)$;
		\item Alexandrov's condition: for any convex set $\omega\varsubsetneq\Sm$, the measure satisfies \newline
		 $\sigma(\omega^*)<\mu(\Sm\setminus\omega)$ where $\omega^*$ is the polar set of $\omega$;
		\item vertex condition: for any $\xi\in\Sm$, the measure satisfies $\mu(\{\xi\})<\frac{1}{2}\sigma(\Sm)$.
	\end{enumerate}
	Moreover, under these assumptions, there is a unique convex body in $\hyp$ whose curvature measure is $\mu$ . 
\end{theorem}

Proving that these three conditions are necessary is much less straightforward to do than in the Euclidean case. For instance, Euclidean convex bodies satisfy (3) as a corollary of (2) and $\mu(\Sm)=\sigma(\Sm) $. The main tool to prove this part and the uniqueness of the underlying convex body is based on the Cauchy-Crofton formula. This result has been extended to Lorentzian space forms by G. Solanes and E.~Teufel \cite{Solanes-Teufel}. In their paper, the authors mainly deal with smooth hypersurfaces; the non-smooth generalization we need, especially in the case with boundary, is proved in Appendix \ref{CCrofton}. 

The proof of the existence part in Theorem \ref{thm-main} is purely variational and completely independent of the uniqueness property. A key point is a strengthened version of Alexandrov's condition, see Propositions \ref{prop-A_alpha} and \ref{prop-Alexandrov}, which, in a way, enables to reduce the proof to the case where $\mu$ is finitely supported. Nevertheless, let us recall the curvature measure is \emph{not} invariant by dilations, thus Alexandrov's argument via Hausdorff compactness does not apply and proving the result only for finitely supported measures is not sufficient. Our proof builds on the method used by the first author in \cite{Bertrand-2}. However, since the total masses of the measures are no more equal, the optimal mass transport approach is no longer useful. On the contrary, Kantorovich's dual problem remains a pertinent tool. A major novelty compared to the Euclidean case is the nonlinearity of the "hyperbolic" Kantorovich dual problem. This new feature requires a completely new approach to find solutions to this variational problem. Finally, we believe our strategy is flexible enough to be useful in other contexts which will be investigated elsewhere.

%
\subsubsection*{Organization of the paper}
In Section \ref{sec-geometry_convex_bodies}, we recall basic results on the convex sets in the hyperbolic and the de Sitter spaces together with a introduction to the duality between convex bodies adapted to our framework. We then define the curvature measure and prove it satisfies conditions (1)-(3) in Theorem \ref{thm-main}. We conclude this part by proving the uniqueness part in Theorem \ref{thm-main}. An ingredient of constant use in this part is a version of the classical Cauchy-Crofton formula in the de Sitter space. This formula is explained and generalised to arbitrary convex bodies in Appendix \ref{CCrofton}.

The existence of a hyperbolic convex body with prescribed curvature as stated in Theorem \ref{thm-main} 
is proved in two steps. First, in Section \ref{sec-Alexandrov_meets_Kantorovich}, we introduce an optimization problem on the sphere depending on the measure $\mu$  and prove that a solution to this problem, if it exists, gives rise to a convex body in $\hyp$ whose curvature measure is $\mu$ (see Section \ref{sec-Convex_to_pairs} and Theorem~\ref{thm-from_NLK_to_transport}). Second, this optimization problem is solved in Section \ref{sec-solving_NLK} for measures satisfying conditions (1)-(3) in Theorem \ref{thm-main}; this step ends the proof of our main theorem. The analysis of the optimization problem relies on fine properties of $c$-concave functions for a well-chosen and non-standard cost function $c$ on $\Sm$. These properties are proved in Appendix \ref{app-analysis}.

%
%
\section{The geometry of convex sets in the hyperbolic and de Sitter spaces} \label{sec-geometry_convex_bodies}
%

In this section we provide the basics  of hyperbolic and de Sitter convex geometry that are used in this paper. 

\subsection{The geometric framework}\label{sec-gsu}

We refer to \cite[Chapter 4, \S\, hyperquadrics]{ONeill} for the proofs of the results mentioned in this part.

%
\subsubsection*{The hyperbolic and de Sitter spaces}
We consider the hyperbolic and de Sitter spaces as hypersurfaces of the Minkowski space. Let $\langle \cdot,\cdot\rangle$ be the Lorentzian inner product on $\R^{m+2}$ defined by:
$$
\langle x,y\rangle = -x_0y_0 + \sum_{k=1}^{m+1} x_ky_k\ \mbox{ where }\ x=(x_0,\dots,x_{m+1}),\ y=(y_0,\dots,y_{m+1}).
$$
The Minkowski space is $\R^{m+2}$ endowed with $\langle \cdot,\cdot\rangle$. The light cone $\sL=\{x\in\R^{m+2}\ |\ \langle x,x \rangle=0\}$, made of \emph{light-like vectors}, divides $\R^{m+2}$ into the domain of \emph{time-like vectors} (for which $\langle x,x \rangle<0$) and the one of {\emph space-like vectors} (for which $\langle x,x \rangle>0$). The \emph{future cone} is $\sF=\{x\in\R^{m+2}\ |\ \langle x,x \rangle\le0\mbox{ and }x_0 \geq 0\}$. 
The hyperbolic space $\hyp$ is
$$
\hyp= \bigl\{x\in\R^{m+2}\ \bigl|\ \langle x,x \rangle=-1\mbox{ and }x_0>0\bigr\},
$$
and the de Sitter space $\dS$ is
$$
\dS= \bigl\{x\in\R^{m+2}\ \bigl|\ \langle x,x \rangle=1\bigr\}.
$$
For $x\in\hyp$, the tangent space at $x$ is $T_x\hyp=x^\bot$, and $\hyp$ is equipped with the restriction of the Lorentzian inner product to each tangent space. Because the points $x\in\hyp$ are all time-like, this turns $\hyp$ into a Riemannian manifold. The same process turns $\dS$ into a Lorentzian manifold.

We identify $\R^{m+1}$ with $\{0\}\times\R^{m+1}\subset\R^{m+2}$; note that the restriction of the Lorentzian inner product to that subspace is the Euclidean inner product. The unit sphere of $\R^{m+1}$ is denoted by $\Sm$ and called the \emph{equator} of $\dS$.

For $x\in\hyp$, the equality $T_x\hyp=x^\bot$ allows us to identify the unit sphere $U_x\hyp$ of $T_x\hyp$ with $\{\xi\in x^\bot\ |\ \langle \xi,\xi \rangle=1\}=x^\bot\cap\dS$. In particular, for $o=(1,0,\dots,0)$, this gives $U_o\hyp=\dS\cap\R^{m+1}=\Sm$. Consequently, $\Sm$ is both the unit tangent sphere of the hyperbolic space at $o$ and the equator of the de Sitter space.

%
\subsubsection*{Geodesics in $\hyp$ and $\dS$} For $x\in\hyp$ and $\xi\in U_x\hyp$, the geodesic $c$ of $\hyp$ starting at $x$ with initial speed $\xi$ is given by $c(t) = \cosh (t) x + \sinh (t) \xi$. It is the intersection of $\hyp$ with the vector 2-plane of $\R^{m+2}$ spanned by $x$ and $\xi$. In what follows, for $\xi\in\Sm$, we will denote by $c_\xi$ the geodesic with initial point $o$ and initial speed $\xi$.

In the de Sitter space, the geodesics are also obtained by intersecting $\dS$ with vector 2-planes, but their parameterization depends on the initial speed. For $x\in\dS$ and $\xi\in T_x\dS$, the geodesic $c$ of $\dS$ starting at $x$ with initial speed $\xi$ is given by
$$
c(t)=\left\{\begin{array}{ll}
	\cos (t)x + \sin (t)\xi & \mbox{ if } \langle \xi,\xi \rangle=1; \\
	x + t\xi & \mbox{ if } \langle \xi,\xi \rangle=0; \\
	\cosh (t)x + \sinh (t)\xi & \mbox{ if } \langle \xi,\xi \rangle=-1.
\end{array}\right.
$$

\begin{remark}
	If $c$ is a geodesic of $\hyp$ with initial point $x$ and initial speed $\xi\in U_x\hyp$, then we have $\xi\in\dS$ and $x\in T_\xi\dS$. Moreover, differentiating the expression of $c$, we get that $c'$ is the geodesic of $\dS$ with initial point $\xi$ and initial speed $x$. In particular, for $\eta\in\Sm$, the derivative of the geodesic $c_\eta$ in $\hyp$ is $c'_\eta$, the geodesic of $\dS$ with initial point $\eta$ and initial speed $o$.
\end{remark}

More generally, the $k$-dimensional complete totally geodesic submanifolds of $\hyp$ are precisely the non-empty sets obtained as the intersection of vector $(k+1)$-planes with $\hyp$. In the de Sitter space, a smooth submanifold is said to be \emph{space-like} if all its tangent spaces are only made of space-like vectors. Similarly to the hyperbolic case, the $k$-dimensional, complete, space-like, totally geodesic submanifolds of $\dS$  are the intersections of space-like vector $(k+1)$-planes with $\dS$.

As a particular case, if $\zeta\in U\hyp$, $H_\zeta=\zeta^\bot\cap\hyp$ denotes the hyperplane of $\hyp$ orthogonal to $\zeta$. In a similar way, if $x\in T\dS$ then $\hat{H}_x=x^\bot\cap\dS$ is the hyperplane of $\dS$ orthogonal to $x$ (note that, if $x$ is time-like, then $\hat{H}_x$ is a space-like hyperplane).

Moreover, the isometry group of $X$, with $X =\hyp$ or $X=\dS$, acts transitively onto the subsets of complete (space-like) totally geodesic submanifolds of $X$ of a given dimension. Note that \emph{all} the totally geodesic submanifolds we consider in this paper are assumed to be complete even if we do not explicitly state so.

Last, we recall that the de Sitter space is homeomorphic to a cylinder $\R \times\Sm$ and the map
\begin{equation}\label{eqn-diffeo_de_Sitter}
	\left\{\begin{array}{rcl}
				\R\times \Sm & \to & \dS \\
				(t,\eta) & \mapsto & c'_\eta(t)
			\end{array}\right.
\end{equation}
is a diffeomorphism. Moreover, the pull-back of the canonical metric through this diffeomorphism is
\begin{equation}\label{eqn-met_de_Sitter}
g_{\dS} = -dt^2 + \cosh^2 (t) \can_{\Sm},
\end{equation}
where $\can_{\Sm}$ stands for the canonical metric on $\Sm$. Using this identification, we write
\begin{equation}\label{eqn-projection_equator}
	Q:\left\{\begin{array}{rcl}
				\dS & \to & \Sm \\
				(t,\eta) & \mapsto & \eta
			\end{array}\right.
\end{equation}
the projection on the equator of $\dS$.

%
\subsection{Convex bodies in $\hyp$ and $\dS_+$}\label{Sec-convBo}
%
In the hyperbolic space $\hyp$, we term \emph{convex body} a compact geodesically convex domain $\Omega\subset\hyp$ with non-empty interior. The forthcoming construction of the curvature measure applies to \emph{pointed} convex bodies. Without loss of generality, from now on we assume this point to be $o=(1,0,\dots,0)\in\R^{m+2}$ and to belong to the interior of $\Omega$. 

In the following, for a given subset $A\subset\R^{m+2}$, we denote by
\begin{equation}\label{eqn-def_cone}
	\sC(A)=\{\lambda x\ |\ \lambda\in\R^+,x\in A\}
\end{equation}
the cone generated by $A$. If $\Omega\subset\hyp$ is a convex body, then $\sC(\Omega)$ is a convex cone, and we 
could also define a convex body in $\hyp$ as the intersection of $\hyp$ with a closed convex cone $\mathscr{C}$ of $\R^{m+2}$ with non-empty interior, whose tip is the origin, and which is strictly contained in the future cone $\sF$ (namely, the intersection of $\mathscr{C}$ with the light cone is reduced to the origin). In order to easily adapt tools from Euclidean convex geometry to hyperbolic geometry, let us also emphasize that the intersection of the cone $\sC(\Omega)$ with the hyperplane $\P =\{1\}\times \R^{m+1}$ is a Euclidean convex body $\Omega_E$ contained in the open unit Euclidean ball; the converse also holds  true since $\Omega=\sC(\Omega_E)\cap\hyp$. Therefore, there is a one-to-one correspondence between hyperbolic convex bodies and Euclidean ones contained in the open unit ball.
 	
For $x\in\dom$ and $\zeta\in U_x\hyp$, the hyperplane $H_\zeta$ orthogonal to $\zeta$ at $x$ is a \emph{support hyperplane to $\Omega$ at $x$} if $\Omega$ is contained  in the closed half-space of $\hyp$ bounded   by $H_\zeta$ and containing $o$. If so, $\zeta$ is said to be a \emph{unit normal vector} at $x$. Given the description of totally geodesic hypersurfaces recalled in the previous section, the above cone construction also induces a one-to-one correspondence between the support hyperplanes to $\Omega$ and those to $\Omega_E$ (since both support hyperplanes uniquely determine a support hyperplane to the underlying convex cone).

Using the description of space-like totally geodesic submanifolds of $\dS$ recalled in Section \ref{sec-gsu}, we follow the previous discussion in order to define the \emph{space-like convex bodies} in $\dS_+= \dS \cap\{x\in\R^{m+2}|x_0>0\}$ that are of constant use in this paper.

\begin{definition}[Space-like convex bodies]
	A set $\hat{\Omega}\subset \dS_+$ is \emph{a space-like convex body} if $\hat{\Omega}=\sC\cap\dS_+$ where $\sC\subset \R^{m+2}$ is a closed convex cone which contains the future cone $\sF$ (without $0$) in its interior.
	
	For $\zeta \in\partial\hat{\Omega}$ and $x\in U_{\zeta}\dS$, the hyperplane $\hat{H}_x$ orthogonal to $x$ at $\zeta$ is a support hyperplane to $\hat{\Omega}$ at $\zeta$ if $\sC(\hat{\Omega})$ and $\{o\}$ are contained in the same half-space of $\R^{m+2}$ defined by $x^\bot$.
\end{definition}

\begin{remark}\label{rem-conv} Notice that 
	\begin{enumerate}
		\item In general, convex bodies in $\dS$ can be defined as intersections of $\dS$ with convex cones of $\R^{m+2}$, the space-like ones being the convex bodies with only space-like support hyperplanes.
		
		To any space-like convex body $\hat{\Omega} \subset \dS_+$ corresponds an $\alpha>0$ such that $x_0\ge\alpha$ for any $(x_0,\dots,x_{m+1})\in\hat{\Omega}$.
		
		\item  As a consequence of the definition, any support hyperplane $\hat{H}_x$ to $\hat{\Omega}$ has to be space-like and we can assume that $x \in \hyp$. Moreover, a space-like convex body is necessarily unbounded, and it is homeomorphic to $[0,+\infty)\times \Sm$. 
		
		\item By considering the intersection of the cone $\sC(\hat{\Omega})\cup\sF$ and the hyperplane $\P =\{1\}\times \R^{m+1}$, we obtain a one-to-one correspondence between the space-like convex bodies $\hat{\Omega}$ in $\dS_+$ and the Euclidean convex bodies $\hat{\Omega}_E$ in $\P$ containing the closed unit ball in their interior. As for hyperbolic convex bodies, there is also a one-to-one correspondence between the support  hyperplanes to $\hat{\Omega}$ and those to $\hat{\Omega}_E$ obtained through the cone generated by the support hyperplane in either case. 

		\item Despite the Euclidean models of $\hyp$ and $\dS_+$ we use are not conformal to  $\hyp$ and $\dS_+$ respectively, if the correspondence between $\Omega$ and $\Omega_E$ maps $x\in\dom$ to $y\in\dom_E$, it also maps the exterior normals to $\dom$ at $x$ to the exterior normals to $\dom_E$ at $y$. This follows from metric considerations. The same holds for a space-like convex body $\hat{\Omega}\subset\dS_+$ and its corresponding $\hat{\Omega}_E\subset \P$. For instance, given $\hat{H}_{c_\eta(t)}$ a totally geodesic hypersurface in $\dS$, $\eta$ is characterized as the unique point where the following function $f$ attains its maximum:
		$$
		f: \zeta \longmapsto \sup \{ s\ |\ \forall a \leq s \  c'_{\zeta}(a) \notin \hat{H}_{c_\eta(t)} \}.
		$$
		The function $f$ is the support function of the space-like convex body $\tilde{\Omega}$ bounded by  $\hat{H}_{c_\eta(t)}$. Thus, to $\tilde{\Omega}$ corresponds $\tilde{\Omega}_E$, bounded by a hyperplane $\hat{H}_E$ (corresponding to $\hat{H}_{c_\eta(t)}$), and its support function is $\text{cotanh}\, h$; $\eta$ is also characterized as the unique minimum of
		$$
		f: \zeta \longmapsto \sup \{ s\ |\  \forall a \leq s \ a\zeta \notin \hat{H}_E \}.
		$$
		
		A similar phenomenon holds for hyperbolic convex bodies. 
	\end{enumerate} 
 \end{remark}

Last, we define the \emph{Gauss map} of a hyperbolic convex body. The Gauss map maps each point $x\in\dom$ to the set of exterior unit normal vectors at $x$. In Euclidean geometry, it is described as a (multivalued) map from $\dom$ to the unit sphere. In non-flat spaces, there is no canonical identification of tangent spaces. This is why we consider the hyperbolic space as an hypersurface of the Minkowski space so that we can identify each unit tangent sphere of $\hyp$ with a subset of the de Sitter space: for all $x \in \hyp$, $U_x\hyp=x^\bot\cap\dS$. The following definition extends the Gauss map defined by E. Teufel for smooth hypersurfaces  \cite[Definition 1]{Teufel} to arbitrary convex bodies.
\begin{definition}[Gauss map of $\Omega$]\label{def-Gauss-map}
	Let $\Omega\subset\hyp$ be a convex body. The Gauss map of $\Omega$ is defined as the multivalued map $G:\dom\rightrightarrows\dS$, where $G(x)\subset\dS$ is the set of outward unit normal vector(s) at $x$.
\end{definition}

In the next paragraph, we introduce the hyperbolic and de Sitter counterparts of the standard radial and support functions. See \cite{Schneider} for more details about the Euclidean framework.

\subsubsection*{Radial and support functions}

We begin with hyperbolic convex bodies. For $\xi\in\Sm$, recall that $c_\xi (t)= \cosh(t) \, o + \sinh(t) \, \xi$ is the geodesic starting at $o$ with initial speed $\xi$.

\begin{definition}\label{def-radial_support_functions}
	Let $\Omega\in\hyp$ be a convex body with the point $o$ in its interior. The radial function $r:\Sm\to ]0,+ \infty[$ and support function $h:\Sm\to ]0,+ \infty[$ of $\Omega$ are defined by
	$$
	r(\xi) = \sup\{t\in \,(0,+\infty)\ |\ c_\xi(t)\in\Omega \},
	$$
	and
	$$
	h(\eta) = \sup\{t\in \,(0,+\infty)\ |\ \Omega\cap H_{c'_\eta(t)}\not=\emptyset \}.
	$$
\end{definition}
In particular, the boundary $\dom$ of $\Omega$ is the radial graph over $\Sm$ of the function $r$, namely the map
$$
P:\left\{\begin{array}{lcl}
\Sm & \longrightarrow & \dom \\
\xi & \longmapsto & \cosh (r(\xi)) \,o + \sinh (r(\xi)) \xi
\end{array}\right.
$$
is a homeomorphism.

Standard trigonometric calculations based on the cone construction recalled above connect the hyperbolic radial and support functions of $\Omega$ to their Euclidean counterparts associated to the convex body $\Omega_E$. These functions are denoted by $r_E$ and $h_E$ respectively. More precisely, the relations are 
\begin{eqnarray}\label{eqn_r_and_h}
    \tanh  h=  h_E \qquad\mbox{and}\qquad \tanh r = r_E \nonumber \\
	\tanh (h(\eta)) =  \max_{\zeta\in\Sm}\bigl(\tanh (r(\zeta))\langle \eta,\zeta\rangle\bigr),
\end{eqnarray}
where the equality in the second line follows from the definition of $h_E$: 
$$
h_E(\eta)= \max_{\zeta\in\Sm}\bigl( r_E(\zeta)\langle \eta,\zeta\rangle\bigr).
$$
The maximum in (\ref{eqn_r_and_h}) may be achieved at more than one point, and we set \\$T:\Sm\rightrightarrows\Sm$ the multivalued map defined by
\begin{equation}\label{eqn-definition_T}
	\xi\in T(\eta) \Leftrightarrow \tanh (h(\eta)) =\tanh (r(\xi))\langle \eta,\xi\rangle.
\end{equation}

We now define the radial and support functions of a space-like convex body in $\dS_+$.
\begin{definition}
	Let $\hat{\Omega}\in\dS_+$ be a space-like convex body. The radial function $\hat{r}:\Sm\to (0,+\infty)$ and the support function $\hat{h}:\Sm\to (0,+\infty)$ of $\hat{\Omega}$ are defined by
	$$
	\hat{r}(\eta) = \sup\{t\in \,(0,+\infty)\ |\ c'_\eta(t)\not\in\hat{\Omega} \},
	$$
	and
	$$
	\hat{h}(\xi) = \sup\{t\in \,(0,+\infty)\ |\ \hat{\Omega}\cap \hat{H}_{c_\xi(t)}=\emptyset \}.
	$$
\end{definition}
In particular, the boundary $\partial\hat{\Omega}$ of $\hat{\Omega}$ is the radial graph over $\Sm$ of the function $\hat{r}$, namely the map $Q$ restricted to $\partial\hat{\Omega}$ is one-to-one and
$$
Q^{-1}:\left\{\begin{array}{lcl}
				\Sm & \longrightarrow & \partial\hat{\Omega} \\
				\eta & \longmapsto & \sinh (\hat{r}(\eta)) \,o + \cosh (\hat{r}(\eta)) \eta
			\end{array}\right.
$$
is a homeomorphism.

%
\subsection{Duality and convex bodies}\label{sec-duality}

For more on the results described in this part, we refer the interested reader to the book \cite{Schneider} for an exposition in the standard Euclidean framework and to the recent paper \cite{Fillastre-Seppi} for a more geometrical discussion, including the hyperbolic and de Sitter spaces among others.

The building block of this part is the duality of convex cones in $\R^{m+2}$ endowed with a non-degenerate bilinear form $b$ that we now recall. Given a cone $\mathscr{C}$ whose tip is $0$, the polar cone $\mathscr{C}^*$ of $\mathscr{C}$ is defined as
$$
\mathscr{C}^*= \{y \in \R^ {m+2}\ |\ \forall x\in\mathscr{C}\ b(x,y) \leq 0 \}.
$$
Elementary considerations show that $\mathscr{C}$ is a convex cone if and only if 
\begin{equation}\label{eq-duCo}
(\mathscr{C}^*)^* =\mathscr{C}.
\end{equation}

When $b$ is the standard Euclidean inner product, the polar transform of cones is strongly related to the polar transform of Euclidean convex bodies with $0$ in their interior. Precisely, given $K$ a Euclidean convex body in $\R^{m+1}$ with $0$ in its interior, the polar body $K^*$ of $K$, defined as
$$
K^*=\{y \in \R^{m+1} \ |\ \forall x\in K\ b(x,y)\leq 1 \},
$$
is a convex body with $0$ in its interior.

By embedding $\R^{m+1}$ as $\{1\}\times \R^{m+1}$ into $\R^{m+2}$ and using the cone over a subset of $\R^{m+2}$ \eqref{eqn-def_cone}, it is straightforward to check that
\begin{equation}\label{eq-dualLorII}
	K^* \sim \{1\} \times K^* = (\sC(\{-1\} \times K))^* \cap \P,
\end{equation}
where $\P= \{1\}\times \R^{m+1}$ as above. We infer from this property and \eqref{eq-duCo} that $(K^*)^*=K$.


From now on, the duality in $\R^{m+2}$ is intended with respect to the Lorentzian inner product $\langle \cdot, \cdot \rangle$. Note that for this choice of bilinear form, we can rewrite \eqref{eq-dualLorII} in a more convenient way  for us:
\begin{equation}\label{eq-dualLor}
	  K^* \sim \{1\} \times K^* = (\sC(\{1\} \times K))^* \cap \P.
\end{equation}

In the same vein, note that if, in addition, $K$ is contained in the open unit ball $B$ then $K^*$ contains the closed unit ball $\overline{B}$ in its interior. The converse also holds true.

According to the previous discussion and using the description of convex bodies given in Section \ref{Sec-convBo}, for $\Omega$ a hyperbolic convex body (resp. $\hat{\Omega}$ a space-like convex body in $\dS_+$), we define its polar body $\Omega^*$ in $\dS_+$ (resp. $\hat{\Omega}^*$ in $\hyp$) as
$$
\Omega^* = \sC( \Omega)^* \cap \dS \qquad\mbox{and}\qquad \hat{\Omega}^* = (\sC(\hat{\Omega}))^* \cap \hyp.
$$ 

Using the correspondence between $\Omega$ and $\Omega_E$ introduced earlier, we can rewrite $\Omega^*$ as
$$
\Omega^* = \sC(\Omega_E)^* \cap \dS= \sC(\Omega_E^*) \cap \dS
$$
where the second equality follows from \eqref{eq-dualLor}. Consequently, by definition of $\Omega_E$, we immediately get  that $\Omega^*$ is a space-like convex body of $\dS_+$. The same argument  
applies when considering the polar body of $\hat{\Omega}$, it gives
$$ \hat{\Omega}^* = \sC(\hat{\Omega}_E)^* \cap \hyp= \sC(\hat{\Omega}_E^*) \cap \hyp$$
and proves that $\hat{\Omega}^*$ is a hyperbolic convex body.

As a consequence, we obtain $(\Omega^*)^* =\Omega$ and $(\hat{\Omega}^*)^*=\hat{\Omega}$. In particular, any space-like convex body in $\dS_+$ is the polar of a unique hyperbolic convex body. Thus, in the rest of the paper, $\Omega^*$ denotes a space-like convex body in $\dS_+$, while $r^*$ and $h^*$ denote the radial and support functions of $\Omega^*$ respectively.

Using the identification $\dS \sim \R \times \Sm$ defined by \eqref{eqn-diffeo_de_Sitter}, and given $\zeta \in \dS$, there exists a unique $(t, \eta)\in \R\times\Sm$ such that $\zeta=c_\eta'(t)=  \sinh (t) \,o+\cosh (t)\, \eta $, and we can rewrite $ \Omega^*$ in a more explicit way:
\begin{align*}
	\Omega^* & = \bigl\{c_\eta'(t)\ |\ \eta\in\Sm,\ t\geq 0,\ \forall\xi\in\Sm\ \forall s\in[0,r(\xi)]\ \ \langle c_\xi(s),c_\eta'(t) \rangle \le 0 \bigr\} \\
	 & = \bigl\{c_\eta'(t)\ |\ \eta\in\Sm ,\ t\geq 0, \forall\xi\in\Sm\ \forall s\in[0,r(\xi)] \\
	 & \qquad \qquad \cosh (t)\sinh (s)\langle\eta,\xi\rangle -\sinh (t) \cosh (s) \le 0 \bigr\} \\
	 & = \bigl\{c_\eta'(t)\ |\ \eta\in\Sm,\ t\geq 0,\ \forall\xi\in\Sm\ \ \tanh (r(\xi))\langle\eta,\xi\rangle \le \tanh (t) \bigr\}  \\
	 & = \bigl\{c_\eta'(t)\ |\ \eta\in\Sm,\ t\in[h(\eta),+\infty) \bigr\},
\end{align*}
where we used the definition of the support function $h$ to get the last equality. In particular, we obtain $\dom^*=\{c_\eta'(h(\eta))\ |\ \eta\in\Sm \}$ and $r^*=h$. 

The latter formula also derives from the relations between radial and support functions of Euclidean polar bodies, denoted by $r_E$ and $h_E$, and $r_E^*$ and $h_E^*$ respectively.
These relations are (see \cite{Schneider})
$$   r_E^*=  1/h_E  \qquad\mbox{and}\qquad  h_E^* = 1/r_E. $$

Combining this together with the relations \eqref{eqn_r_and_h} and easy trigonometric computations makes it clear that $h=r^*$ and $r=h^*$.

\begin{remark}
	The last important property of duality we will use is the following. In Euclidean space, to a support hyperplane $H$ to $\Omega_E$ corresponds a unique $x \in \partial \Omega_E^*$ such that $\langle x, H\rangle=0$ and vice versa. Namely, if $H$ is orthogonal to $\eta \in \Sm$, then $x=r_E^*(\eta) \eta$. The point $x$ can also be characterized using the cone construction above:
	   $\{x\}=\sC(H)^\bot\cap\P$ while $ H=\sC(\{x\})^\bot\cap\P.$
	
	The same computations and geometric construction apply to (space-like) convex bodies in $\hyp$ and $\dS_+$. It gives that
	\begin{itemize}
		\item to a support hyperplane $\hat{H}_{c_\xi (h^*(\xi))}$ to $\Omega^*$ corresponds $x=c_\xi (r(\xi)) \in \partial \Omega$ and vice versa,
		\item  to a support hyperplane $H_{c'_\eta (h(\eta))}$ to $\Omega$ corresponds $x=c'_\eta (r^*(\eta)) \in \partial \Omega^*$ and vice versa.
	\end{itemize}
\end{remark}

By combining this together with the correspondence of normal vectors explained in Remark \ref{rem-conv} (4), we get the following result.
\begin{proposition}\label{prop-S_T_equality_case}
	Let $\Omega\subset\hyp$ be a convex domain with $o$ in its interior and $\Omega^*\subset\dS$ be its polar. The Gauss map $G:\dom\to\dS$ satisfies
	\begin{equation}\label{e_Gauss_onto}
	G(\partial \Omega) =  \dom^*.
	\end{equation}
	Moreover, the multivalued map $S:\Sm\rightrightarrows\Sm$ defined by
	$S=Q\circ G\circ P$ satisfies
	$$
	\eta \in S(\xi) \Leftrightarrow  \tanh (h(\eta)) = \tanh (r(\xi))\langle \eta,\xi \rangle \Leftrightarrow h_E(\eta)=r_E(\xi)\langle \xi,\eta) \Leftrightarrow\xi \in T(\eta).
	$$
\end{proposition}

\begin{remark}\label{rem_id_T}
The last equivalence above shows that $T$ is the map sending $\eta$, an exterior unit normal vector to $\Omega_E$, to the directions $\xi$ pointing towards the intersection of the support hyperplane orthogonal to $\eta$ and the convex body $\Omega_E$. In the next part, it is recalled that such a $\xi$ is unique for $\sigma$-a.e. $\eta$.

\end{remark}

%
\section{The curvature measure}
%

\subsection{Definition of the curvature measure}

In the Euclidean case, the Gauss map $G$ of a smooth convex body maps onto the unit sphere and pushes the curvature measure $Kdv_{\dom}$ forward to the uniform measure $\sigma$ on $\Sm$. Identifying the unit sphere with $\dom$ via the radial homeomorphism $P:\Sm\to\dom$, the map $G\circ P$ is a transport map from the sphere into itself pushing the (pulled-back)  curvature measure of $\Omega$ to the canonical measure $\sigma$. For non-smooth convex bodies, thanks to the regularity properties of $G\circ P$ ({\it i.e.} $G\circ P$ admits an inverse function defined $\sigma$-a.e. on $\Sm$), this transportation property is used to define the curvature measure $\mu_E$  \cite[Chap. 1 \S5]{Bakelman}:
\begin{equation}\label{e_def_Euclid_curv}
\mu_E := \sigma (G \circ P(\cdot)).
\end{equation}

In the particular case of polytopes, it gives rise to a measure consisting in a linear combination of Dirac masses supported in the directions towards the vertices with weights equal to the exterior solid angles.

The optimal transport approach to solve Alexandrov's problem is based on this construction of the curvature measure. It consists, for a given measure $\mu$ on $\Sm$ and a suitable cost function, in finding the map $T= (G\circ P)^{-1}$ as the unique optimal transport map pushing $\sigma$ forward to $\mu$.

In the hyperbolic case, the definition is similar though more involved. It is based on the area measure $\sigma_{\partial \Omega^*}$ of the boundary of the polar body $\Omega^*$, whose existence is to be proved. Indeed, $\dS$ is only a Lorentzian manifold, not all hypersurfaces admit an area measure, in particular issues can occur at points where the tangent space contains light-like vectors. We study the area measure in the next part. 

\subsubsection*{Area measure of $\partial \Omega^*$}

Throughout this part, $\Omega$ denotes a convex body in $\hyp$ with $o$ in its interior. Our approach to define the area measure of the polar body $\Omega^*$ consists in mimicking the one used for a submanifold $N$ of a Riemannian manifold when $N$ is the graph of a smooth function.

Recall that $\partial \Omega^*$ has only space-like supporting hyperplanes (on which the induced metric is then Riemannian). Moreover, the boundary $\partial \Omega^*$ is the range of the mapping $Q^{-1}: \Sm \longrightarrow \dS \sim \R\times \Sm$ defined by $Q^{-1}(\eta) = (h(\eta), \eta)$. Let us also remind the reader that the support function $h$ is related to its Euclidean counterpart $h_E$ by the formula
$$ h = \argth \circ \, h_E.$$

Consequently, since the convex body $\Omega_E$ is contained in the open unit ball, we infer from the above formula that $h$ is a Lipschitz function. Therefore, in order to check the area measure is well-defined, it suffices to prove that the Jacobian determinant of $Q^{-1}:\Sm\to\dom^*$ is non-zero $\sigma$-a.e.
\begin{lemma}
	Let $\Omega$ as above. Then, for $\sigma$-a.e. $\eta$, the Jacobian determinant of $Q^{-1}:\Sm\to\dom^*$ satisfies
	$$
	|\det T_{\eta}Q^{-1}| = \frac{\cosh^{m+1} (h(\eta))}{\cosh (r(T(\eta)))} \neq 0.
	$$
	
	As a consequence, the area measure of $\dom^*$ is well-defined and satisfies  
	\begin{equation}\label{e_def_AreaMea}
		\sigma_{\partial \Omega^*} = Q^{-1}_\# \left( \frac{\cosh^{m+1} (h)}{\cosh (r\circ T)} \;\sigma\right)
	\end{equation} 
\end{lemma}
\begin{proof}
	Recall that $h$ is a Lipschitz function therefore differentiable $\sigma$-a.e. on $\Sm$. In what follows, $\nabla h$ denotes its gradient relative to the canonical metric on the sphere.

	Let $\eta$ be a point where $h$ is differentiable. To compute the Jacobian determinant, we use that the restriction of $\langle \cdot, \cdot \rangle$ to $\dS \sim \R \times \Sm$ is $g_{\dS}= -dt^2 + \cosh^2(t)\,\can_{\Sm}$. Therefore, if we set $(e_1,\cdots,e_m)$ an orthonormal basis of $T_{\eta} \Sm$, we get
	$$
	g(T_{\eta}Q^{-1} (e_i), T_{\eta}Q^{-1} (e_j)) = - \langle \nabla h(\eta),e_i \rangle \langle \nabla h(\eta),e_j \rangle + \cosh^2 (h(\eta))\,\langle e_i, e_j\rangle.
	$$ 

	To compute the Jacobian determinant we use the standard fact that, given $M$ a row matrix of size $m$ and $I_m$ the identity matrix of size $m\times m$, the determinant of $I_m-M^t \times M$ is $1 - | M|^2$, where $|\cdot|$ stands for the Euclidean norm. This yields
	$$
	{\det}^2 (T_{\eta}Q^{-1}) = \cosh^{2m} (h(\eta)) \left(1 - \frac{|\nabla h|^2}{\cosh^2 (h(\eta))}\right).
	$$

	To conclude, it remains to prove that the right-hand side is equal to the one stated in the Lemma. To this aim, note that if $\xi \in T(\eta)$, we have $\tanh (h(\eta))= \tanh (r(\xi)) \langle \eta, \xi\rangle$ \eqref{eqn-definition_T}. Since $\tanh (h(\eta'))-\tanh (r(\xi)) \langle \xi,\eta'\rangle\geq 0$ holds for arbitrary $\eta'\in\Sm$, differentiating this expression at $\eta$, we infer, for any $\zeta\in T_\eta\Sm$,
	$$
	\frac{\langle \nabla h(\eta),\zeta \rangle}{\cosh^2 (h(\eta))} = \tanh (r(\xi))\, \langle \xi , \zeta\rangle.
	$$

	This yields
	$$
	\nabla  h(\eta) = \cosh^2 (h(\eta))\tanh (r(\xi))\, \left(\xi -\langle\eta, \xi\rangle \eta\right) \in T_{\eta} \Sm.
	$$
	In particular $\xi$ is uniquely determined and coincides with $T(\eta)$. Using \eqref{eqn-definition_T} once again, we can compute the Jacobian of $Q^{-1}$:
	\begin{align*}
		1 - \frac{|\nabla h|^2 (\eta)}{\cosh^2 (h(\eta))} & = 1- \cosh^2(h(\eta)) \tanh^2(r(\xi))(1- \langle \xi, \eta\rangle^2) \\
        & = 1 + \cosh^2(h(\eta)) \tanh^2(h(\eta)) -\cosh^2(h(\eta))\tanh^2(r(\xi))\\
        & = 1 +\sinh^2(h(\eta)) -\cosh^2(h(\eta))\tanh^2(r(\xi)) \\
        & = \cosh^2(h(\eta)) (1 -\tanh^2(r(\xi))) \\
        & = \frac{ \cosh^2(h(\eta))}{\cosh^2(r(\xi))}>0.
	\end{align*}
\end{proof}

We end this part with a result on the area measure $\sigma_{\partial \Omega^*}$ in the smooth and polytope cases.
\begin{proposition}\label{prop-curvature_measure}
	Let $\Omega$ be a convex body of $\hyp$ with $o$ in its interior.
	\begin{enumerate}
		\item If $\dom$ is $C^2$ and strictly convex then $\dom^*$ is a $C^1$ space-like hypersurface of $\dS$ and $\sigma_{\dom^*}=dv_{\dom^*} =G_{\#} (Kdv_{\dom})$.
		\item Let $\xi\in\Sm$ and $x=P(\xi)\in\dom$. Then, $\sigma_{\dom^*}(G(x))$ is the exterior solid  angle of $\dom$ at $x$. In particular, if $\Omega$ is a convex polytope, the volume $\sigma_{\partial \Omega^*}(\dom^*)$ is the sum of the exterior angles over the vertices. 
	\end{enumerate}
\end{proposition}
\begin{proof}
	Let us prove (1). By assumption on $\Omega$, the Gauss map is injective and $C^1$. Therefore, $\dom^*$ is a $C^1$-manifold and $\Omega^*$ is strictly convex. Consequently, the maps $G,P,$ and $Q$ are $C^1$-diffeomorphisms between $C^1$-Riemannian manifolds. If $f: (A, dv_A) \rightarrow (B,dv_B)$ is a $C^1$-diffeomorphism from $A$ onto $B$, the change of variable formula yields
	\begin{equation}\label{e_tec_coi}
		f_{\#} (|\det Tf| dv_A) = dv_B.
	\end{equation}
	
	We infer from the above formula and $K = \det TG$ that $dv_{\dom^*} =G_{\#} (Kdv_{\dom})$. To conclude, we note that \eqref{e_def_AreaMea} combined with the fact $\dom^*$ is $C^1$ entails $ \sigma_{\dom^*}=dv_{\dom^*}$.
	
	Let us now prove (2). First recall that $G(x)$ is the set of outward unit normal vectors at $x$, so that the exterior angle of $\dom$ at $x$ is $\sigma_x(G(x))$, where $\sigma_x$ is the canonical measure of $U_x\hyp$ (induced by the restriction of $\langle \cdot,\cdot\rangle$ to the tangent space $T_x\hyp$). Therefore, (2) is proved if we check that the area measure of $G(x)$, which is contained in a totally geodesic hypersurface $H$ of $\dS$, is the restriction of the standard Riemannian area mesure induced by $\langle \cdot,\cdot\rangle$. But this hypersurface $H$ equipped with the restriction of $\langle \cdot,\cdot\rangle$ is in particular a $C^1$ Riemannian manifold. Thus (\ref{e_tec_coi}) above  with $f = Q^{-1}: \Sm \rightarrow H$ applies and gives us the result. 
\end{proof}

\subsubsection*{Definition of the curvature measure}
%
With the notion of area measure at our disposal, we can now define the curvature measure of $\Omega$ in a similar fashion than in the Euclidean case.
\begin{definition}\label{def-curvature_measure}
	Let $\Omega$ be a convex body of $\hyp$ with $o$ in its interior. The curvature measure of $\Omega$ is the measure $\mu$ on $\Sm$ defined by
	$$ \mu = \sigma_{\partial \Omega^*} (G \circ P (\cdot)),$$
	where $\sigma_{\partial \Omega^*}$ is the area measure of $\partial \Omega^*$.
\end{definition}

According to \eqref{e_Gauss_onto}, $G(\partial \Omega ) =\dom^*$, thus we infer from the definition the total mass of $\mu$:
$$
\mu(\Sm) =  \sigma_{\partial \Omega^*} (\partial \Omega^*) =: |\partial \Omega^*|.
$$

As a corollary of Proposition \ref{prop-curvature_measure}, we get $\mu=P^{-1}_\#(Kdv_{\dom})$ if $\dom$ is $C^2$ and $\mu=\sum\alpha_i\delta_{\xi_i}$ if $\Omega$ is a polytope, where $\xi_i$ are the directions pointing to the vertices and $\alpha_i$ are the exterior angles at $P(\xi_i)$.

Using the properties of the area measure, we can also characterize $\mu$ implicitly by the following formula.
\begin{lemma}\label{def_CurvMeas_Impli}
	Let $\Omega\subset \hyp$ be a convex body with $o$ in its interior. The curvature measure $\mu$ is characterized by the formula
	$$
	T_\# (\cosh^{m+1} (h) \, \sigma) = \cosh (r) \, \mu.
	$$ 
	\end{lemma}
\begin{proof}
	First note that $\cosh(r)>0$ on $\Sm$, thus $\mu$ is uniquely determined by the above formula. Let us set $f(\eta)=\frac{\cosh^{m+1}(h(\eta))}{\cosh(r(T(\eta)))}$. We obtain the result by first plugging (\ref{e_def_AreaMea}):
	$$
	\sigma_{\partial \Omega^*} = Q^{-1}_\# (f(\eta) \,\sigma)
	$$
	into the definition of the curvature measure:
	$$
	\mu = \sigma_{\partial \Omega^*} (G \circ P (\cdot)).
	$$
	
	We get $ \mu =  \left( f\,\sigma\right) (Q \circ G \circ P (\cdot))= \left( f\,\sigma\right) (S (\cdot))$ by definition of $S$. Now, according to Remark \ref{rem_id_T}, we can rewrite this equality as $ \mu =  \left( f\,\sigma\right) (T^{-1} (\cdot))$. We complete the proof thanks to the formula $T_\# ((a \circ T) \,m) =  a \,T_\#(m)$.
\end{proof}

\begin{remark}
	As explained in the introduction, our approach to Alexandrov's problem in $\hyp$ relies on optimal transport theory. Lemma \ref{def_CurvMeas_Impli} is the transport property of the curvature measure we will use in the next sections.
	
	Since $\sigma$ and $\mu$ do not have the same total mass, a normalization is necessary for a transport map to exist. This normalization depends on $\Omega$. This is the main difference between the Euclidean and hyperbolic cases.
\end{remark}

\begin{remark}\label{curv_non_homo}
	Using the relations involving the radial and support functions of $\Omega$ and its Euclidean counterpart $\Omega_E$, the above formula can be rewritten as
	$$ \mu = T_\# \left( \sqrt{\frac{1-r_E^2(T(\eta))}{(1-h_E^2(\eta))^{m+1}}}\, \sigma\right).$$
	This formula highlights the erratic behaviour of the curvature measure with respect to Euclidean dilations. 
\end{remark}	
	
\begin{remark}\label{rem-curvature_measures_coincide}
	For a smooth convex body $\Omega$, the family of curvature measures on $\dom$ are defined using the symmetric functions of the principal curvatures and the volume form $dv_{\dom}$. These measures appear in the Steiner formula giving the volume of parallel sets to $\Omega$. P. Kohlmann extended the definitions of these measures to arbitrary convex bodies and proved a Steiner-like formula  \cite[Theorem 2.7]{Kohlmann}, defining in particular the Gauss curvature as a measure supported in $\dom$.

	Our definition of the curvature measure $\mu$ coincides with Kohlmann's one in the sense that $P_\#\mu=\Phi_0(\Omega,.)$, where $\Phi_0(\Omega,.)$ is the Gauss curvature measure as defined in \cite{Kohlmann}. To see this, consider a sequence $(\Omega_k)_{k\in\N}$ of smooth strictly convex domains converging to $\Omega$ whose existence is granted by Proposition~ \ref{prop-approximation_convex}. We have seen during the proof of the above proposition that $\mu = T_\# (f\sigma)$ where $f\sigma = Q_\# (\sigma_{\partial \Omega^*})$. Thus, by combining these properties we get $\mu = (T \circ Q)_\#(\sigma_{\partial \Omega^*})$. If $\Omega$ is further assumed to be smooth and strictly convex, Proposition~\ref{prop-curvature_measure} yields $P_\# \mu =  Kdv_{\dom}$. Now, Steiner formula  yields
	\begin{equation} \label{eqn-coincidence_curv_measures}
		(P_k)_\#\mu_k = K_kdv_{\dom_k} = \Phi_0(\Omega_k,.),
	\end{equation}
	where $\mu_k$ is the curvature measure of $\Omega_k$, $P_k:\Sm\to\dom_k$ is the radial homeomorphism and $K_k$ is the Gauss curvature of $\dom_k$.

	Thanks to Proposition \ref{prop-approximation_convex}, we have $\mu_k\rightharpoonup\mu$ and the radial functions $r_k$ converge uniformly to $r$ on $\Sm$. Therefore, the radial projections $P_k$ converge uniformly to $P$, and we get $(P_k)_\#\mu_k\rightharpoonup P_\#\mu$. On the other hand, using Proposition~\ref{prop-approximation_convex} (3) and \cite[Theorem 2]{Veronelli}, we also obtain $\Phi_0(\Omega_k,.)\rightharpoonup\Phi_0(\Omega,.)$. Passing to the weak limit in \eqref{eqn-coincidence_curv_measures} finally gives the result.

\end{remark}
	
\subsubsection*{The total curvature of a convex body}
We call $\mu(\Sm)$ the \textit{total curvature} of $\Omega$. A priori, it is different from $\sigma(\Sm)$: for example, Gauss-Bonnet formulas imply $\mu(\Sbb^1) = 2\pi + |\Omega|$ for a smooth convex body in $\Hbb^2$, and $\mu(\Sbb^2) = 4\pi + |\dom|$ for a smooth convex body in $\Hbb^3$. Based on Gauss-Bonnet-Chern formulas, similar results exist in higher dimensions, they involve the principal curvatures of $\dom$. In particular, these formulas imply 
\begin{equation}\label{e_CompaTotMa}
	\mu(\Sm)>\sigma(\Sm)
\end{equation}
for a smooth convex body (cf. for example \cite[Theorem 1]{Solanes-2}). In the next section, (\ref{e_CompaTotMa}) is proved for arbitrary convex bodies. 

\begin{remark}\label{rem-dependance_base_point}
	The curvature measure depends on the base point $o$ only through the homeomorphism $P:\Sm\to\dom$. In particular, changing the basepoint (or equivalently, moving $\Omega$ by an isometry of $\hyp$, keeping the point $o$ in the interior) does not change the total mass of the curvature measure.
\end{remark}

%
\subsection{Properties of the curvature measure}\label{sec-necessary_conditions}
%
Not all finite measures on $\Sm$ are the curvature measure of a convex body. In the Euclidean case, A.D. Alexandrov gave a necessary and sufficient condition for a measure with the same total mass as $\sigma$ to be the curvature measure of a convex body \cite{Alexandrov}. In this section, we discuss the conditions satisfied by the curvature measure of a hyperbolic convex body.

To this aim, we first establish a monotonicity property regarding the total area measure of space-like convex bodies in the de Sitter space.

\begin{proposition}\label{le_monoto_measu} Let $\Omega^*_1, \Omega^*_2$ be two space-like convex bodies in $\dS_+$. Then, $\Omega^*_1 \supset \Omega^*_2$ implies
$$ |\partial \Omega^*_1| \leq |\partial \Omega^*_2|$$
and equality occurs if and only if $\Omega^*_1= \Omega^*_2$.
\end{proposition}

\begin{proof}
	The proof follows from  the Cauchy-Crofton formula, proved in Theorem~\ref{thm-Crofton-conv}, applied to $\Omega^*_1$ and $\Omega^*_2$, with $\omega=\{\eta\in\Sm\ |\ h_1(\eta)<h_2(\eta)\}$ and
	$\Sigma_i=\{(h_i(\eta),\eta)\ |\ \eta\in\omega\}$. Since $\Omega^*_2 \subset \Omega^*_1$ we have $h_1\le h_2$ on $\Sm$ and $h_1=h_2$ on $\omega^c$. The Cauchy-Crofton formula gives
	$$
	|\partial \Omega^*_2| - |\partial \Omega^*_1|  = |\Sigma_2|-|\Sigma_1| = \frac{m}{|\Sbb^{m-1}|}\int_{\cL_s}(\#(\gamma\cap\Sigma_1)-\#(\gamma\cap\Sigma_2))\,\dl(\gamma).
	$$
	This implies the above integral is non-negative, and vanishes only if $\omega= \emptyset$ namely $\Omega^*_1= \Omega^*_2$.
\end{proof}
%
\subsubsection*{The total curvature condition}
\begin{proposition}\label{prop-CN_total_curvature}
	If $\Omega\subset\hyp$ is a convex body with the point $o$ in its interior then its curvature measure $\mu$ satisfies $\mu(\Sm)>\sigma(\Sm)$.
\end{proposition}
\begin{proof}
	Let us term "ball" a convex body in the de Sitter space whose boundary is determined by the equations $h=r=\varepsilon>0$, $\varepsilon$ being the radius of the ball (this body is the polar of the ball $B(o,\varepsilon)\subset\hyp$). By assumption, $\Omega^*$ is strictly contained in a ball of sufficiently small radius $\varepsilon >0$ (cf. Remark \ref{rem-conv}(1)). The curvature measure of this ball is, by definition,  $Q^{-1}_\#(\cosh^m (\varepsilon) \, \sigma)$. To conclude, it suffices to apply Proposition \ref{le_monoto_measu} to $\Omega^*$ and the ball of radius $\varepsilon$.
\end{proof}

\subsubsection*{Alexandrov's condition}
Let $\cF$ denote the set of non-empty closed sets of $\Sm$ and $\cC\subset\cF$ the subset of convex sets. In this part, all the measures we consider on $\Sm$ are assumed to have a total mass greater than or equal to $\sigma(\Sm)$.

The polar of a convex $\omega\in\cC$ is $\omega^*=\{\zeta\in\Sm\ |\ \forall\xi\in\omega\ \  \langle\xi,\zeta\rangle\le 0\}$. Since $\langle\xi,\zeta\rangle=\cos(d(\xi,\zeta))$, where $d$ is the intrinsic distance on $\Sm$, then $\omega^*=\Sm\setminus\omega_{\hpi}$, where $\omega_{\hpi}$ stands for the open $\hpi$-neighborhood of $\omega$.
\begin{remark}\label{rem-dual_convex_Gauss_map}
	This notion of polar set in the sphere is related to the Gauss map in the same way as in the Euclidean case: for a point $x\in\dom$, the set of unit tangent vectors at $x$ pointing to the interior of $\Omega$ is a convex subset of $U_x\hyp$ whose polar in $U_x\hyp$ is exactly $G(x)$.
\end{remark}

In the Euclidean case, provided that $\mu(\Sm)=\sigma(\Sm)$, the necessary and sufficient condition for a measure to be the curvature measure of a convex set is:
\begin{definition}\label{def-Alexandrov}
	A measure $\mu$ on $\Sm$ satisfies \textit{Alexandrov's condition} if for any convex $\omega\in\cC$ with $\omega\not=\Sm$,
	$$
	\sigma(\omega^*) < \mu(\Sm\setminus\omega).
	$$
\end{definition}
As observed by A.D. Alexandrov in \cite{Alexandrov_book}, this condition is also satisfied by the curvature measure of a convex polyhedron in $\Hbb^3$. In the next proposition, we prove this condition holds for arbitrary convex bodies in the hyperbolic space.
\begin{proposition}\label{prop-CN_Alexandrov}
	Given $\Omega$  a convex body of $\hyp$ with the point $o$ in its interior, the curvature measure $\mu$ of $\Omega$ satisfies Alexandrov's condition.
\end{proposition}
\begin{proof}
	Let $\omega\varsubsetneq\Sbb^m$ be a closed convex set. Since the intersection of $\Omega$ with a totally geodesic submanifold going through $o$ is a convex body with $o$ in its relative interior, an easy induction on the dimension $m$ reduces the proof to the case where $\omega$ has non-empty interior.
	
	Consider the convex cone $\Gamma_\omega\subset\hyp$ defined by
	$$
	\Gamma_\omega = \{\exp_o(t\xi)\ |\ \xi\in\omega,\ t\in\R_+\},
	$$
	and set $\Omega_1=\Omega\cap\Gamma_\omega$. Since $o$ belongs to the interior of $\Omega$, $\Omega_1\varsubsetneq\Omega$. Morever we can apply a well-chosen isometry to $\Omega$ and $\Omega_1$ so that the origin belongs to the interior of their images. As noticed in Remark \ref{rem-dependance_base_point}, the total curvature of a hyperbolic convex body is preserved by isometry. Therefore, the monotonicity formula gives us
	$$|\dom_1^*| < |\dom^*|.$$ 
	The right-hand side is the total curvature $\mu(\Sm)$ of $\Omega$. By definition of $\Omega_1$ and noting that the curvature of $\Omega_1$ at the vertex $o$ is $\sigma(\omega^*)$, we get 
	$$|\dom_1^*| \geq \mu(\omega) + \sigma(\omega^*)$$
	which completes the proof.	
\end{proof}

Alexandrov condition can be sharpened, using the compactness of $\cF$ and $\cC$ for the Hausdorff distance:
\begin{proposition}\label{prop-A_alpha}
	Let $\mu$ be a measure on $\Sm$ with $\mu(\Sm)\ge\sigma(\Sm)$. The following are equivalent:
	\begin{enumerate}
		\item $\mu$ satisfies Alexandrov's condition.
		\item there exists $\alpha>0$ such that for any convex $\omega\in\cC$ with $\omega\not=\Sm$, $\mu$ satisfies
		$\sigma(\omega^*) + \alpha \le \mu(\Sm\setminus\omega)$.
	\end{enumerate}
\end{proposition}
\begin{proof}
	See \cite[Proposition 3.7]{Bertrand-2} where the result is proved when $\mu(\Sm)=\sigma(\Sm)$, the same proof applies under our assumptions.
\end{proof}
For $\alpha>0$ we will say that a measure $\mu$ satisfies the condition $(A_\alpha)$ if it satisfies condition (2) above. This condition can be stated in terms of open neighborhoods of closed sets in $\Sm$. For $C\in\cF$ and a number $\rho>0$, $C_\rho=\{\zeta\in\Sm\ |\ d(\zeta,C)<\rho\}$ denotes the open $\rho$-neighborhood of $C$.
\begin{proposition}\label{prop-Alexandrov}
	Let $\mu$ be a measure on $\Sm$ such that $\mu(\Sm)\ge\sigma(\Sm)$. The following are equivalent:
	\begin{enumerate}
		\item there exists $\alpha>0$ such that $\mu$ satisfies $(A_\alpha)$.
		\item there exists $\beta>0$ such that, for any closed set $C\in\cF$,
		
		$\mu(C)\le\sigma(C_{\frac{\pi}{2}-\beta}) + \mu(\Sm)-\sigma(\Sm)$. 
		\item there exists $\beta>0$ such that, for any closed set $C\in\cF$, 
		
		$\sigma(C)\le\mu(C_{\frac{\pi}{2}-\beta})$.
	\end{enumerate}
	Moreover, for any $\beta>0$ there exists $\alpha>0$ (only depending on $\beta$) such that any measure satisfying (2) or (3) with $\beta$ satisfies $(A_\alpha)$.
\end{proposition}
\begin{proof}
	For any $\rho>0$ and any $C\in\cF$, we have $(\Sm\setminus C_\rho)_\rho \subset \Sm\setminus C$. Using this and considering the closed set $\Sm\setminus C_{\hpi-\beta}$, the following is easy to check for any $\beta>0$ 
	$$
	\mu(C)\le\sigma(C_{\frac{\pi}{2}-\beta}) + \mu(\Sm)-\sigma(\Sm)\ \  \Leftrightarrow\ \ \sigma(C)\le\mu(C_{\frac{\pi}{2}-\beta}),
	$$
	which proves that (2) and (3) are equivalent.
	
	Assuming that (1) is satisfied by $\mu$, we prove (2) following \cite[Proposition 3.9]{Bertrand-2}. For $s\ge0$ and $C\in\cF$, consider $f_s(C)=\sigma(C_{\hpi-s})-\mu(C)$. We want to prove that $f_\beta\ge\sigma(\Sm)-\mu(\Sm)$ for some $\beta>0$. If $\conv(C)\not=\Sm$, using $C^*=\conv(C)^*$, $C\subset\conv(C)$, and $(A_\alpha)$, we get
	\begin{eqnarray}
		f_0(C) & = & \sigma(\Sm) - \sigma(C^*) - \mu(C) \nonumber \\
		 & \ge & \sigma(\Sm) + \alpha - \mu(\Sm\setminus\conv(C)) - \mu(C) \nonumber \\
		 & > & \sigma(\Sm)-\mu(\Sm). \nonumber
	\end{eqnarray}
	If $\conv(C)=\Sm$ then $C_\hpi=\Sm$ and $f_0(C)\ge\sigma(\Sm)-\mu(\Sm)$. The rest of the proof is identical to that of \cite[Proposition 3.9]{Bertrand-2}. 
		
	Assume that (3) is satisfied for some $\beta>0$. We will show that (1) is satisfied for some $\alpha>0$ depending only on $\beta$ (and not on $\mu$), proving the equivalence of the three conditions together with the last statement. Let $\omega\in\cC\setminus \{ \Sm\}$. 	For $s\in[0,\frac{\beta}{2}]$, consider $f(s)=\sigma(\omega_{\hpi-\beta+s})$. Using the coarea formula we infer that $f$ is differentiable a.e. and absolutely continuous; moreover $f'(s)=|\partial\omega_{\hpi-\beta+s}|\ge Is(\sigma(\omega_{\hpi-\beta+s}))$, where $Is(v)=\inf\{|\partial D|\ |\ \sigma(D)=v\}$ is the isoperimetric profile of $\Sm$.
	
	Since $\omega$ is convex, non-empty and different from $\Sm$, there exist $\eta_0$ and $\xi_0$ such that $\eta_0\in\omega\subset B(\xi_0,\hpi)$. Therefore, its $(\hpi-\beta+s)$-neighborhood satisfies 
	
	\begin{center} $B(\eta_0,\hpi-\beta)\subset \omega_{\hpi-\beta+s}\subset B(\xi_0,\pi-\frac{\beta}{2})$,\end{center} 
	and $v(\hpi-\beta)\le \sigma(\omega_{\hpi-\beta+s}) \le v(\pi-\frac{\beta}{2})$, where $v(\rho)$ is the volume of a ball of radius $\rho$ in $\Sm$. Defining $I_\beta=\min\{Is(v(\hpi-\beta)),Is(v(\pi-\frac{\beta}{2}))\}$, the concavity property of $Is$ \cite{Bayle} yields $f'(s)\ge I_\beta >0$ for a.e. $s\in[0,\frac{\beta}{2}]$. Therefore
	$$
	\sigma(\omega_\hpi) \ge f(\frac{\beta}{2}) \ge f(0) + \frac{\beta I_\beta}{2} = \sigma(\omega_{\hpi-\beta}) + \frac{\beta I_\beta}{2}.
	$$
	For $\alpha= \frac{\beta I_\beta}{2}$ and thanks to (3), we get
	\begin{eqnarray}
		\alpha + \sigma(\omega^*) & = & \alpha +\sigma(\Sm) - \sigma(\omega_\hpi) \nonumber \\
		 & \le & \sigma(\Sm) - \sigma(\omega_{\hpi-\beta}) \nonumber \\
		 & \le & \mu((\Sm\setminus\omega_{\hpi-\beta})_{\hpi-\beta}) \nonumber \\
		 & \le & \mu(\Sm\setminus\omega). \nonumber
	\end{eqnarray}
\end{proof}

Given $\beta>0$, we will say that $\mu$ satisfies $(B_\beta)$ if it satisfies (2) or (3) with $\beta$ in the above proposition.

%
\subsubsection*{The vertex condition}
There is another condition satisfied by the curvature measure of convex bodies. Because a convex body has non-empty interior and is bounded, the exterior angle is always less than $\frac{1}{2}\sigma(\Sm)$.
\begin{definition} \label{def-vertex-condition}
	A measure $\mu$ on $\Sm$ satisfies the \textit{vertex condition} if for any point $\xi\in\Sm$, $\mu(\{\xi\})<\frac{1}{2}\sigma(\Sm)$.
\end{definition}
\begin{proposition}
	Let $\Omega$ be a convex body of $\hyp$ with $o$ in its interior. Then, its curvature measure $\mu$ satisfies the vertex condition.
\end{proposition}
\begin{proof}
	Let $\xi\in\Sm$ and set $x=P(\xi)\in\dom$. The set $\omega$ of unit vectors in $U_x\hyp$ pointing towards the interior of $\Omega$ has non-empty interior in $U_x\hyp$. Therefore, its polar $\omega^*$ is a closed set of  $U_x\hyp$ contained in an open hemisphere. Following Remark \ref{rem-dual_convex_Gauss_map} and Proposition \ref{prop-curvature_measure}, we get $\mu(\{\xi\}) = \sigma_x(\omega^*) < \frac{1}{2}\sigma(\Sm)$.
\end{proof}

Notice that the vertex condition does not appear in the Euclidean setting. Indeed, it is a consequence of Alexandrov's condition combined with  $\mu(\Sm)=\sigma(\Sm)$ (use Definition \ref{def-Alexandrov}  with $\omega=\{\xi\}$).

%
\subsection{Uniqueness of the convex body with prescribed curvature} \label{sec-Uniqueness}
%
The goal of this part is to prove the following result which implies the uniqueness statement in the main Theorem.

\begin{theorem}
	Let $\Omega_1$ and $\Omega_2$ be two hyperbolic convex bodies with $o$ in their interior. Assume that $\Omega_1$ and $\Omega_2$ have the same curvature measure. Then, $\Omega_1=\Omega_2$.
\end{theorem}
\begin{proof}
	The proof is by contradiction. Suppose $\Omega_1\neq \Omega_2$ have the same curvature measure $\mu$. Thus, according to the monotonicity of the total curvature proved in Proposition~\ref{le_monoto_measu}, neither $\Omega_1 \subsetneq \Omega_2$ nor $\Omega_2 \subsetneq \Omega_1$ can hold, otherwise the total curvatures of $\Omega_1$ and $\Omega_2$ would not be equal. Consequently, the open set $\omega=\{\eta \in \Sm|\, h_1(\eta) <h_2(\eta)\}$ is non-empty, and for $\Sigma_i=\{(h_i(\eta),\eta)\ |\ \eta\in\omega\}$, Theorem \ref{thm-Crofton-conv} yields  
	\begin{equation}\label{e-uni-proof}
		|\Sigma_2|-|\Sigma_1|= \frac{m}{|\Sbb^{m-1}|}\int_{\cL_s}(\#(\gamma\cap\Sigma_1)-\#(\gamma\cap\Sigma_2))\dl(\gamma)>0.
	\end{equation}
	For $i \in \{1,2\}$, let $T_i$ and $S_i$ be the mappings relative to $\Omega_i$ introduced in Section~\ref{sec-geometry_convex_bodies}. Now, let us introduce $B_1=\{\xi \in \Sm\ |\ S_1(\xi) \subset \omega\}$ and $B_2=T_2(\omega)$. Combining the properties of the mappings $S_1$ and $T_2$ (in particular the fact they are onto) together with the compactness of $\Sm$ yield that $B_1$ and $B_2$ are measurable sets. More precisely, $B_1$ is an open set while $B_2= (T_2(\omega^c))^c \cup (T_2(\omega)\cap T_2(\omega^c))$ is the union of an open set and a $\sigma$-negligible set.
	
	By definition of these sets, 
	$$
	\mu(B_1) = \sigma_{\dom_1^*} (G_1 \circ P_1 (B_1)) \leq |\Sigma_1|,
	$$
 	while
	$$
	\mu(B_2) = \sigma_{\dom_2^*} (G_2 \circ P_2 (B_2)) \geq |\Sigma_2|.
	$$
	Therefore, \eqref{e-uni-proof} yields $\mu(B_1) < \mu (B_2)$.
	
	The rest of the proof consists in showing $B_2 \subset B_1$ which leads to a contradiction with the previous estimate. Let $\xi_0 \in B_2$. Then, there exists $\eta \in \omega$ such that $\tanh h_2(\eta)=\tanh r_2(\xi_0)\langle \xi_0, \eta \rangle$, note that $\langle \xi_0, \eta \rangle>0$. Since $h_1<h_2$ on $\omega$, this yields $r_2(\xi_0) > r_1(\xi_0)$.
	
	Now, let $\zeta\in S_1(\xi_0)$. We must show $\zeta\in\omega$, namely $ h_1(\zeta)<h_2(\zeta)$. By definition of the support function,
	$$
	\tanh h_2(\zeta) \geq \tanh r_2(\xi_0)\langle\xi_0,\zeta\rangle > \tanh r_1(\xi_0)\langle\xi_0,\zeta\rangle
	$$
	since $\langle\xi_0,\zeta\rangle>0$ but $\tanh r_1(\xi_0)\langle\xi_0,\zeta\rangle = \tanh h_1(\zeta)$ by definition of $\zeta$. This completes the proof of $B_2 \subset B_1$ and contradicts the hypothesis $\Omega_1\neq \Omega_2$.
\end{proof}

%
%
\section{The prescription of curvature as an optimization problem} \label{sec-Alexandrov_meets_Kantorovich}
%
%
In this section we relate Alexandrov's problem to an optimization problem on $\Sm$ by noticing that the radial and support functions of a convex set give rise, by a simple transformation, to a pair of $c$-conjugate functions. This observation was first made by V. Oliker in the Euclidean case \cite{Oliker-2}. However, the optimization problem we get significantly differs from the Euclidean one because of the densities involved in Proposition \ref{def_CurvMeas_Impli}. We refer the reader to Appendix \ref{app-analysis} for a brief reminder on $c$-conjugate functions and related tools.
%
\subsection{From convex bodies to $c$-conjugate pairs} \label{sec-Convex_to_pairs}
%
Let $\Omega$ be a convex body in $\hyp$ with $o$ in its interior. From  \eqref{eqn_r_and_h} and Proposition \ref{prop-S_T_equality_case} we obtain, for any $\eta,\xi\in\Sm$,
$$
-\ln(\tanh h(\eta)) + \ln(\tanh r(\xi)) \le -\ln(\langle\eta,\xi\rangle)
$$
with equality if and only if $\eta\in S(\xi)( \Leftrightarrow \xi\in T(\eta))$. Defining the cost function $c$ on $\Sm\times \Sm$ by
\begin{equation} \label{eqn-cost_function}
	c(\eta,\xi)=\left\{\begin{array}{rl}
	-\ln(\langle\eta,\xi\rangle) & \mbox{ if } \langle\eta,\xi\rangle >0 \\
	+\infty & \mbox{ otherwise}
	\end{array}\right.,
\end{equation}
and writing $\varphi=-\ln(\tanh h)$, $\psi=\ln(\tanh r)$, the above inequality is equivalent to
$$
\varphi(\eta) + \psi(\xi) \le c(\eta,\xi)
$$
with equality if and only if $\eta\in S(\xi) (\Leftrightarrow \xi\in T(\eta))$. Therefore, 
$$
\varphi(\eta)=\inf\{ c(\eta,\xi)-\psi(\xi)\ |\ \xi\in\Sm\}\ \ \mbox{ and }\ \ \psi(\xi)=\inf\{c(\eta,\xi)-\varphi(\eta)\ |\ \eta\in\Sm\}
$$
which is equivalent to $\varphi=\psi^c$ and $\psi=\varphi^c$. Moreover, according to Definition \ref{def-c-concavity} and the remark afterwards, their $c$-superdifferentials satisfy $\partial^c\varphi=T$ and $\partial^c\psi=S$. Thus, $(\varphi,\psi)$ is a $c$-conjugate pair and the curvature measure of $\Omega$ is characterized (see Lemma \ref{def_CurvMeas_Impli})   by
\begin{equation}\label{eqn-phi_psi_mu}
(\partial^c\varphi)_\#\biggl(\frac{\sigma}{(1-\ex^{-2\varphi})^\frac{m+1}{2}}\biggr) = \frac{\mu}{\sqrt{1-\ex^{2\psi}}},
\end{equation}
where this equality makes sense because $(\partial^c\varphi)$ coincides $\sigma$-a.e with a measurable function.

Conversely, given $(\varphi,\psi)$ a pair of $c$-conjugate functions on $\Sm$ such that $\varphi>0$ and $\psi<0$, let
$$
r = \frac{1}{2}\ln\Bigl(\frac{1+\ex^\psi}{1-\ex^\psi}\Bigr) \qquad\mbox{and}\qquad h=\frac{1}{2}\ln\Bigl(\frac{1+\ex^{-\varphi}}{1-\ex^{-\varphi}}\Bigr),
$$
 we obtain
$$
\tanh h(\eta) \ge \tanh r(\xi)\langle\eta,\xi\rangle
$$
with equality if and only if $\xi\in\partial^c\varphi(\eta) (\Leftrightarrow \eta\in\partial^c\psi(\xi))$. Therefore, according to the Euclidean theory for which $h_E=\tanh h$ and $r_E= \tanh r$, $\Omega$ is a convex set with radial function $r$, support function $h$, and Gauss map $\partial^c\psi\circ P^{-1}$.

Now, Alexandrov's problem can be rephrased as a transport problem in the following way:
\begin{Alex-pb}
	Given a measure $\mu$ on $\Sm$ such that $\mu(\Sm)>\sigma(\Sm)$, $\mu$ satisfies Alexandrov's and vertex conditions, does there exist a $c$-conjugate pair $(\varphi,\psi)$ with $\varphi>0$, $\psi<0$, and \eqref{eqn-phi_psi_mu} holds?
\end{Alex-pb}
%
\subsection{A nonlinear Kantorovich problem}
%
Let $F:I\to\R$ and $G:J\to\R$ be two $C^1$ functions on open real intervals $I$ and $J$, where $J=-I=\{-x|\ x\in I\}$. Assume further that $F$ and $G$ are nondecreasing so that their derivatives $f=F'$ and $g=G'$ are nonnegative. For measurable functions $\varphi:\Sm\to I$ and $\psi:\Sm\to J$ we consider the functional
$$
\cK(\varphi,\psi) = \int_{\Sm} F(\varphi)d\sigma + \int_{\Sm} G(\psi)d\mu.
$$

For later use, let us introduce the following nonlinear Kantorovich problem:
\begin{NLK-pb}
	Find a pair $(\bar{\varphi},\bar{\psi})$ such that
	\begin{equation}
		\cK(\bar{\varphi},\bar{\psi}) = \max\left\{ \cK(\varphi,\psi)\ |\ (\varphi,\psi)\in\cA\right\},
	\end{equation}
	where the set of admissible pairs is
	\begin{multline*}
	\cA = \bigl\{ (\varphi,\psi)\in \cB(\Sm)^2\ \bigl|\ \varphi:\Sm\to I,\ \psi:\Sm\to J,\\
		F(\varphi)\in L^1(d\sigma),\ G(\psi)\in L^1(d\mu),\ 
		\forall\xi,\eta\ \ \varphi(\eta)+\psi(\xi)\le c(\eta,\xi)\bigr\},
	\end{multline*}
	where $\cB(\Sm)$ is the space of Borel functions on $\Sm$, and $c$ is defined by \eqref{eqn-cost_function}.
\end{NLK-pb}

The solutions to this problem, whenever they exist, are expected to be $c$-conjugate functions. The fact that $J=-I$ and the properties of $c$-conjugate pairs (see Proposition \ref{prop-min_c-concave} (2)) ensure that admissible $c$-conjugate pairs do exist.

In this part, we prove
\begin{theorem}\label{thm-from_NLK_to_transport}
	If $(NLK)$ admits a maximizing pair $(\varphi,\psi)\in\cA$ which is $c$-conjugate, then the $c$-superdifferential $\partial^c\varphi$  satisfies
	\begin{equation}\label{eqn-Euler-Lagrange}
		\partial^c\varphi_\#(f(\varphi)\sigma) = g(\psi)\mu,
	\end{equation}
	where  $\partial^c\varphi$ coincides $\sigma$-a.e. with a measurable map, $f=F'$, and $g=G'$.
\end{theorem}
\begin{proof}
	Let $(\varphi,\psi)\in\cA$ be a maximizing pair for the problem $(NLK)$ and assume that $\varphi=\psi^c$ and $\psi=\varphi^c$.
	
	Equation \eqref{eqn-Euler-Lagrange} is the Euler-Lagrange equation of  $\cK$ and derives from the computation of its derivative. For a continuous function $\theta:\Sm\to\R$, and $s$ small enough, the image of $\psi+s\theta$ lies in the open set $J$, therefore it makes sense to consider $\cK((\psi+s\theta)^c,\psi+s\theta)$ and its derivative $\frac{d}{ds}_{|_{s=0}}\cK((\psi+s\theta)^c,\psi+s\theta)$.
	
	\noindent\textbf{Claim 1.} \textit{$(\psi+s\theta)^c$ converges uniformly to $\varphi$ when $s\to 0$.}
	
	For any $\eta$, $\xi$ in $\Sm$, we have
	$$
	c(\eta,\xi)-\psi(\xi)-|s|\|\theta\|_\infty \le c(\eta,\xi)-\psi(\xi)-s\theta(\xi) \le c(\eta,\xi)-\psi(\xi)+|s|\|\theta\|_\infty,
	$$
	and taking the infimum on $\xi\in\Sm$, we get
	$$
	\|(\psi+s\theta)^c - \varphi \|_\infty \le |s|\|\theta\|_\infty
	$$
	which proves Claim 1.
	
	\noindent\textbf{Claim 2.} \textit{$\frac{d}{ds}_{|_{s=0}}\int_{\Sm}G(\psi+s\theta)d\mu=\int_{\Sm}\theta(\xi)g(\psi(\xi))d\mu(\xi)$.}
	
	Since $\psi$ and $\theta$ are continuous on $\Sm$ and $g=G'$ is continuous on $J$,  $\frac{\partial}{\partial s}G(\psi+s\theta)=\theta g(\psi+s\theta)$ is uniformly bounded on $\Sm$. Therefore, by differentiating $\int_{\Sm}G(\psi+s\theta)d\mu$ we get Claim 2.	
	
	\noindent\textbf{Claim 3.} \textit{$\frac{d}{ds}_{|_{s=0}}\int_{\Sm}F((\psi+s\theta)^c)d\mu=-\int_{\Sm}\theta(T\eta)f(\varphi(\eta))d\sigma(\eta)$, where $T=\partial^c\varphi$ $\sigma$-a.e..}
	
	Let $\eta\in\Sm$ be such that $\partial^c\varphi(\eta)= T(\eta)$ is single-valued. From the definition of $c$-transform and $c$-superdifferential, we get $\varphi(\eta)+\psi(T(\eta))=c(\eta,T(\eta))$ and $(\psi+s\theta)^c(\eta)\le c(\eta,T(\eta))-\psi(T(\eta))-s\theta(T(\eta))$. Therefore
	\begin{equation} \label{eqn-EL-1}
		(\psi+s\theta)^c(\eta) - \varphi(\eta) \le -s\theta(T(\eta)).
	\end{equation}
	On the other hand, for any $s\not=0$ there exists $\xi_s\in\Sm$ such that
	\begin{equation} \label{eqn-EL-2}
		(\psi+s\theta)^c(\eta) = c(\eta,\xi_s) - \psi(\xi_s) - s\theta(\xi_s).
	\end{equation}
	By combining \eqref{eqn-EL-1}, \eqref{eqn-EL-2}, and  $\varphi(\eta)+\psi(\xi_s)\le c(\eta,\xi_s)$, we obtain
	\begin{equation} \label{eqn-EL-3}
		s(\theta(T(\eta))-\theta(\xi_s)) \le (\psi+s\theta)^c(\eta)-\varphi(\eta) + s\theta(T(\eta)) \le 0.
	\end{equation}
	
	Since $\psi$ and $\theta$ are bounded, \eqref{eqn-EL-2} implies that $c(\eta,\xi_s)$ remains bounded and there exists $\varepsilon>0$ such that $\xi_s\in B(\eta,\frac{\pi}{2}-\varepsilon)$ for any $s$. If $\lim_{k\to+\infty}\xi_{s_k}=\xi$ for some sequence $(s_k)_{k\in\N}$, using Claim 1 and passing to the limit in \eqref{eqn-EL-2} gives $\varphi(\eta) + \psi(\xi) = c(\eta,\xi)$ which in turn implies $\xi=T(\eta)$.  The continuity of $\theta$ and \eqref{eqn-EL-3} yield
	\begin{equation} \label{eqn-EL-4}
	 (\psi+s\theta)^c(\eta) = \varphi(\eta) - s\theta(T(\eta)) + o(s),
	\end{equation}
and, since $F$ is $C^1$, we get
	\begin{equation} \label{eqn-EL-5}
		F((\psi+s\theta)^c(\eta)) = F(\varphi(\eta)) - s\theta(T(\eta))f(\varphi(\eta)) + o(s).
	\end{equation}
	Finally, since $f=F'$ is continuous while $\varphi$, $\psi$, and $\theta$ are bounded, there exists a constant $C$ such that for any $s$ 
	$$
	\Bigl|\frac{F((\psi+s\theta)^c(\eta)) - F(\varphi(\eta))}{s}\Bigl| \le C\frac{|(\psi+s\theta)^c(\eta) - \varphi(\eta)|}{|s|} \le C\|\theta\|_\infty,
	$$
	where  the last inequality follows from \eqref{eqn-EL-2}. Claim 3 is then a consequence of \eqref{eqn-EL-4} and the dominated convergence theorem, since $\partial^c\varphi$ is single-valued $\sigma$-a.e. according to Proposition \ref{prop-min_c-concave}.
	
	We are now in position to prove the theorem. Since $(\varphi,\psi)\in\cA$ is a maximizing pair, for any continuous function $\theta:\Sm\to\R$ the map $s\mapsto\cK((\psi+s\theta)^c,\psi+s\theta)$ attains its maximum at $s=0$. Using Claims 2 and 3 to compute its derivative, we get
	$$
	\int_{\Sm}\theta(\xi)g(\psi(\xi))d\mu(\xi) = \int_{\Sm}\theta(T(\eta))f(\varphi(\eta))d\sigma(\eta)
	$$
	for any continuous  $\theta$. This clearly  implies \eqref{eqn-Euler-Lagrange}.
\end{proof}

%
\subsection{Alexandrov meets Kantorovich}
%
In view of Theorem \ref{thm-from_NLK_to_transport} and \eqref{eqn-phi_psi_mu}, we consider the functions $f:]0,+\infty[\to\R$ and $g:]-\infty,0[\to\R$ defined by
$$
f(u) = \frac{1}{(1-\ex^{-2u})^\frac{m+1}{2}} \qquad\mbox{ and }\qquad g(v)=\frac{1}{\sqrt{1-\ex^{2v}}}.
$$
In order to write the optimization problem, we fix
\begin{equation} \label{eqn-defining_F_G}
	F(u)=\int_{u_0}^u \frac{ds}{(1-\ex^{-2s})^\frac{m+1}{2}} \qquad\mbox{ and }\qquad G(v)=v-\ln(1+\sqrt{1-\ex^{2v}})
\end{equation}
which satisfy $F'=f$ and $G'=g$.

As a corollary of the results in Sections 3.1 and 3.2, the existence of a solution to the hyperbolic Alexandrov problem  follows from that to the nonlinear Kantorovich problem associated to $F$ and $G$ above, provided this solution is a $c$-conjugate pair. We shall prove the existence of such a maximizing pair in the next section. Before that, we give some basic properties of  $F$, $G$, and $c$ that will be used in the proof.

Observe that the cost function of this Kantorovich problem is a convex function of the distance on $\Sm$. Namely, according to \eqref{eqn-cost_function},
$$
c(\eta,\xi)=\left\{\begin{array}{rl}
							-\ln(\cos d(\eta,\xi)) & \mbox{ if } d(\eta,\xi) < \hpi \\
							+\infty & \mbox{ otherwise}
						\end{array}\right..
$$

\begin{proposition}\label{prop-properties_F_G}
	The functions $F$ and $G$ are increasing, $F$ is concave and $G$ is convex. Moreover the following properties hold
	\begin{enumerate}
		\item For any $v\in (-\infty,0)$, $G(v)\le v$, and $u_0$ can be chosen in such a way that for any $u\in(0,+\infty)$, $F(u)\le u$.
		\item $G(-t) \sim_0 -\sqrt{2t}$.
		\item For any $0<R<\hpi$,  $\displaystyle \int_0^t\int_{B(\xi,R)}f(c(\eta,\xi)+s)d\sigma(\eta) ds \sim_0 \frac{\sigma(\Sm)}{2}\sqrt{2t} $.
	\end{enumerate}
\end{proposition}
\begin{proof}
	The monotonicity of $F$ and $G$ is straightforward, the concavity of $F$ (resp, convexity of $G$) derives from  $f'\leq 0$ (resp. $g'\geq 0$).
	
	To prove (1), note the inequality involving $G$ is obvious. Regarding $F$, it is easy to see that $u\mapsto F(u)-u$ is non-decreasing;  besides, writing
	$$
	F(u)-u=\int_{u_0}^u \biggl(\frac{1}{(1-\ex^{-2s})^\frac{m+1}{2}}-1\biggr)ds - u_0,
	$$
	we get
	$$
	\lim_{u\to+\infty}F(u)-u=\int_{u_0}^\infty \biggl(\frac{1}{(1-\ex^{-2s})^\frac{m+1}{2}}-1\biggr)ds - u_0=l(u_0). 
	$$
	Since $l(u_0)$ depends continuously on $u_0$ while $\lim_{u_0\to+\infty}l(u_0)=-\infty$ and $\lim_{u_0\to 0}l(u_0)=+\infty$, we can choose $u_0$ such that $\lim_{u\to+\infty}F(u)-u=0$ and the inequality follows.
	
	Item (2) derives easily from the definition of $G$.
	
	To prove (3), let $I(R,s)=\int_{B(\xi,R)}f(c(\eta,\xi)+s)d\sigma(\eta)$. Writing the measure $\sigma$ in normal coordinates centered at $\xi$, it holds 
	\begin{eqnarray} \label{eqn-F-G-1}
		I(R,s) & = & |\Sbb^{m-1}| \int_0^R \frac{\sin^{m-1}r dr}{(1-\ex^{-2s}\cos^2r)^\frac{m+1}{2}} \nonumber \\
		 & = & |\Sbb^{m-1}| \frac{\ex^{(m+1)s}}{\sqrt{\ex^{2s}-1}} \int_0^{\frac{\sin R}{\sqrt{\ex^{2s}-1}}} \frac{u^{m-1}du}{(1+u^2)^\frac{m+1}{2}\sqrt{1-(\ex^{2s}-1)u^2}} \nonumber \\
		 & = &  |\Sbb^{m-1}| \frac{\ex^{(m+1)s}}{\sqrt{\ex^{2s}-1}} \int_0^{+\infty} \frac{u^{m-1}}{(1+u^2)^\frac{m+1}{2}}\frac{\ind_{[0,\frac{\sin R}{\sqrt{\ex^{2s}-1}}]}(u)}{\sqrt{1-(\ex^{2s}-1)u^2}}du
	\end{eqnarray}
	where we set $u=\frac{\sin r}{\sqrt{\ex^{2s}-1}}$ to get the second equality. Using $0<R<\hpi$, the dominated convergence theorem yields
	\begin{equation} \label{eqn-F-G-2}
		\lim_{s\to 0} \int_0^{+\infty} \frac{u^{m-1}}{(1+u^2)^\frac{m+1}{2}}\frac{\ind_{[0,\frac{\sin R}{\sqrt{\ex^{2s}-1}}]}(u)}{\sqrt{1-(\ex^{2s}-1)u^2}}du = \int_0^{+\infty} \frac{u^{m-1}du}{(1+u^2)^\frac{m+1}{2}}.
	\end{equation}
	Moreover, setting $u=\tan v$ we obtain
	$$
	\int_0^{+\infty} \frac{u^{m-1}du}{(1+u^2)^\frac{m+1}{2}} = \int_0^\hpi \sin^{m-1}v\,dv = \frac{\sigma(\Sm)}{2|\Sbb^{m-1}|},
	$$
	and combining this together with \eqref{eqn-F-G-1} and \eqref{eqn-F-G-2}, we get $I(R,s)\sim_0\frac{\sigma(\Sm)}{2\sqrt{2s}}$. This imply that $\int_0^tI(R,s)ds$ is well-defined and $\int_0^tI(R,s)ds\sim_0 \frac{\sigma(\Sm)}{2}\sqrt{2t}$
thanks to l'Hospital rule. \end{proof}

\begin{proposition}\label{prop-cost_L1}
	For any $\xi\in\Sm$, the function $\eta\mapsto c(\eta,\xi)$ is in $L^1(B(\xi,\hpi),\sigma)$ and $\int_{B(\xi,\hpi)} c(\eta,\xi)d\sigma(\eta)$ does not depend on $\xi$.
\end{proposition}
\begin{proof}
	Using normal coordinates centered at $\xi$, we get
	$$
	\int_{B(\xi,\hpi)} c(\eta,\xi)d\sigma(\eta) = -|\Sbb^{m-1}|\int_0^\hpi \ln(\cos r)\sin^{m-1}rdr
	$$
	which is easily seen to be finite and does not depend on $\xi$.
\end{proof}

In conclusion, a solution to Alexandrov's problem can be obtained by using the functional described in this part. This is the aim of the next part.  
%
%
\section{Solving the optimization problem} \label{sec-solving_NLK}
%
%
In this section, we consider
$$
\cK(\varphi,\psi) = \int_{\Sm} F(\varphi)d\sigma + \int_{\Sm} G(\psi)d\mu,
$$
where $F$ and $G$ are defined by \eqref{eqn-defining_F_G}.

We are going to solve: 
\begin{NLK-pb}
	Find a $c$-conjugate pair $(\bar{\varphi},\bar{\psi})$ such that
	\begin{equation}
	\cK(\bar{\varphi},\bar{\psi}) = \max\left\{ \cK(\varphi,\psi)\ |\ (\varphi,\psi)\in\cA\right\},
	\end{equation}
where	\begin{multline*}
		\cA = \bigl\{ (\varphi,\psi)\in \cB(\Sm)^2\ \bigl|\ \varphi:\Sm\to (0,+\infty),\ \psi:\Sm\to (-\infty,0),\  \varphi\in L^1(d\sigma),\\
			\forall\xi,\eta\ \ \varphi(\eta)+\psi(\xi)\le c(\eta,\xi)\bigr\},
	\end{multline*}
 $\cB(\Sm)$ denotes the space of Borel functions on $\Sm$, and $c$ is defined by \eqref{eqn-cost_function}.
\end{NLK-pb}

We further assume that $\mu$ satisfies $\mu(\Sm)>\sigma(\Sm)$ and Alexandrov's and vertex conditions (see Definitions \ref{def-Alexandrov} and \ref{def-vertex-condition}). 

Since $G\le 0$ and $F$ is sublinear, $\cK(\varphi,\psi)<+\infty$ holds whenever $\varphi\in L^1(d\sigma)$. Moreover, the measures $\sigma$ and $\mu$ being finite, the functional $\cK$ is not identically $-\infty$ on $\cA$,  for example $\cK(1,-1)>-\infty$. Therefore, the problem $(NLK)$ is well-posed.

\begin{theorem}\label{thm-Solving_NLK}
	Let $\mu$ be a measure on $\Sm$ such that $\mu(\Sm)>\sigma(\Sm)$. If $\mu$ satisfies Alexandrov's and vertex conditions then $(NLK)$ has a maximizing pair $(\varphi,\psi)\in\cA$. Moreover, any maximising pair is a $c$-conjugate pair.
\end{theorem}

The proof relies on an intricate compactness argument needed to prove that a maximizing sequence admits a converging subsequence in the $C^0$ topology. A priori, the limit of this sequence is a pair $(\varphi,\psi)$ of continuous functions with $\varphi\ge 0$ and $\psi\le 0$. But for Theorem~\ref{thm-from_NLK_to_transport} to apply, we need $\varphi>0$ and $\psi<0$.Thus, the proof splits in two parts. First, we consider the slightly relaxed problem where $\varphi$ and $\psi$ are allowed to vanish, that is the problem of maximizing $\cK$ over
\begin{multline*}
	\cA' = \bigl\{ (\varphi,\psi)\in \cB(\Sm)^2\ \bigl|\ \varphi:\Sm\to [0,+\infty),\ \psi:\Sm\to (-\infty,0],\  \varphi\in L^1(d\sigma),\\
	\forall\xi,\eta\ \ \varphi(\eta)+\psi(\xi)\le c(\eta,\xi)\bigr\}.
\end{multline*}
We prove this problem admits a maximizing pair $(\tilde{\varphi},\tilde{\psi})$. Then, we prove that any maximizing pair of the relaxed problem actually belongs to $\cA$, completing the proof of Theorem \ref{thm-Solving_NLK}.

In order to build a converging subsequence, we approximate $\mu$ by discrete measures with the same properties as $\mu$. This is enabled by the strengthened Alexandrov condition that we now recall. According to Proposition \ref{prop-Alexandrov}, there exists $\alpha>0$ such that for any convex $\omega$ of $\Sm$ with $\omega\not=\Sm$, 
$$
\sigma(\omega^*) + \alpha \le \mu(\Sm\setminus\omega).
$$

%
\subsection{There exists an optimal pair in $\cA'$}
%
Let $(\tilde{\varphi}^k,\tilde{\psi}^k)_{k\in\N}$ be a maximizing sequence relative to $(NLK)$, where $\tilde{\varphi}^k\ge 0$ and $\tilde{\psi}^k\le 0$. Since $\cK$ is not identically $-\infty$, we may assume  that $-\infty < \int_{\Sm}G(\psi_k)d\mu$. Because $G\le 0$, $G(\psi_k)\in L^1(d\mu)$ and $\psi_k$ must be finite $\mu$-a.e.. As a consequence of Alexandrov's condition, any point on the sphere is at distance less than $\hpi$ of $\spt(\mu)$, therefore the hypotheses of Proposition \ref{prop-continuity-c-transform} are satisfied. Since double convexification does not decrease $\tilde{\varphi}^k$ and $\tilde{\psi}^k$ (see Remark \ref{rem-double_convexification}), and consequently $\cK(\tilde{\varphi}^k,\tilde{\psi}^k)$, we can further assume that $(\tilde{\varphi}^k,\tilde{\psi}^k)_{k\in\N}$ is a sequence of $c$-conjugate pairs. In particular, according to Proposition \ref{prop-continuity-c-transform}, the functions $\tilde{\varphi}^k$ and $\tilde{\psi}^k$ are Lipschitz.

%
\subsubsection*{Discretization of the measure $\mu$} For $\varepsilon>0$ consider an $\varepsilon$-net $(\zeta_1,\dots,\zeta_N)$ of $\spt(\mu)$, namely a finite family of points in $\spt(\mu)$ such that $\spt(\mu)\subset\cup_{i=1}^NB(\zeta_i,2\varepsilon)$, and  $B(\zeta_i,\varepsilon)\cap B(\zeta_j,\varepsilon)=\emptyset$ whenever $i\not= j$.

Let $(W_1^o,\dots,W_N^o)$ be the Vorono\"{\i} domains of $\Sm$ associated to this net, namely, for all $i=1,\dots,N$,  $W_i^o=\{\xi\in\Sm\ |\ \forall j\not=i,\  d(\xi,\zeta_i)<d(\xi,\zeta_j)\}$. Then, define the  partition $(W_1,\dots,W_N)$ of $\Sm$ from the Vorono\"{\i} domains: $W_1=\overline{W_1^o}$, for $i=2,\dots,N$, $W_i=\overline{W_i^o}\setminus(W_1\cup\dots\cup W_{i-1})$.

For any $k\in\N$ and $i=1,\dots,N$, let $\xi_i^k$ be such that $\tilde{\psi}^k(\xi_i^k)=\max_{\overline{W}_i\cap\spt(\mu)}(\tilde{\psi}^k)$, and let $\mu_k$ be the discrete measure defined by
$$
\mu_k=\sum_{i=1}^N a_i\delta_{\xi_i^k}
$$
where $a_i=\mu(W_i)$. Notice that the definition of  $(W_i)_{1\leq i \leq N}$ yields $a_i>0$ for all $i=1,\dots,N$ and $\mu(\Sm)=\sum_{i=1}^N a_i=\mu_k(\Sm)$.

\begin{lemma}\label{lem-discretization_mu}
	For $\varepsilon$ small enough, there exist $\alpha>0$ and $\beta>0$ such that the following holds:
	\begin{enumerate}
		\item for any $k\in\N$, any $i=1,\dots,N$, and any $\xi\in W_i$,  $d(\xi,\xi_i^k)\le \hpi-\beta$.
		\item for any $k\in\N$, the measure $\mu_k$ satisfies Alexandrov's condition $(A_\alpha)$.
	\end{enumerate}
\end{lemma}
\begin{proof}
	Since the measure $\mu$ satisfies Alexandrov's condition, there exists $\gamma>0$ such that, for any $\xi\in\Sm$, $d(\xi,\spt(\mu))\le\hpi-\gamma$. Considering $\xi\in W_i$, there exist $\zeta\in\spt(\mu)$ such that $d(\xi,\zeta)\le\hpi-\gamma$ and $j\in\{1,\dots,N\}$ such that $d(\zeta,\zeta_j)<2\varepsilon$. Therefore, $d(\xi,\zeta_j)\le\hpi-\gamma+2\varepsilon$ and $ W_i\subset \overline{W_i^o}  $ implies $d(\xi,\zeta_i)\le\hpi-\gamma+2\varepsilon$. From $\xi_i^k\in\overline{W}_i\cap\spt(\mu)$, we infer $d(\zeta_i,\xi_i^k)\le 2\varepsilon$ hence $d(\xi,\xi_i^k)\le\hpi-\gamma+4\varepsilon$. This proves (1) with $\beta=\gamma-4\varepsilon$ and $\varepsilon$ small enough. 
	
	According to Proposition \ref{prop-Alexandrov}, there exists $\gamma>0$ such that, for any $C\in\cF$, $	\sigma(C) \le \mu(C_{\hpi-\gamma})$. Fix such a set $C$  and assume $\sigma(C)>0$. Let $i\in\{1,\dots,N\}$ be such that $W_i\cap\spt(\mu)\cap C_{\hpi-\gamma}\not=\emptyset$. By what precedes, for any $k\in\N$,  $\xi_i^k\in C_{\hpi-\beta}$. Thus, for $I=\{i\ |\ W_i\cap\spt(\mu)\cap C_{\hpi-\gamma}\not=\emptyset\}$ and $J=\{i\ |\ \xi_i^k\in C_{\hpi-\beta}\}$,  we have checked that $I\subset J$ and
	$$
	\sigma(C) \le \mu(C_{\hpi-\gamma}) \le \mu\Bigl(\bigcup_{i\in I}W_i\cap\spt(\mu)\Bigr) = \sum_{i\in I}a_i \le \sum_{i\in J}a_i = \mu_k(C_{\hpi-\beta}).
	$$
	Therefore, according to Proposition \ref{prop-Alexandrov}, there exists $\alpha>0$ such that, for any $k\in\N$, $\mu_k$ satisfies Alexandrov's condition $(A_\alpha)$.
\end{proof}

Note that in the above statement the numbers $\alpha$ and $\beta$ does not depend on $k$. Therefore, for any $k\in\N$ the measures $\mu_k$ satisfy 
\begin{equation}\label{eqn-Alexandrov_mu_k}
	\sigma(\omega^*) + \alpha \le \mu_k(\Sm\setminus\omega)
\end{equation}
for any convex set $\omega\subset\Sm$ such that $\omega\not=\Sm$.

For all $i\in\{1,\dots,N\}$ and $k\in\N$, define $\psi_i^k=\tilde{\psi}^k(\xi_i^k)$, and consider the $N$ sequences of real numbers $\psi_1=(\psi_1^k)_{k\in\N},\dots,\psi_N=(\psi_N^k)_{k\in\N}$. By definition, $\psi_i^k=\max_{\overline{W}_i\cap\spt(\mu)}(\tilde{\psi}^k)$; by combining this together with Lemma \ref{lem-discretization_mu} (1) and Proposition \ref{prop-min_c-concave} (3),  we infer that the functions  $\tilde{\psi}^k$ are uniformly bounded from below provided that the sequences $\psi_1,\dots,\psi_N$ are.

Once this bound is established, we get a converging subsequence of the  $(\tilde{\psi}^k)_{k \in \N}$ thanks to Proposition \ref{prop-compacity}. The needed estimate on the sequences $\psi_1,\dots,\psi_N$ is proved in the next section.

%
\subsubsection*{Comparing sequences} A key point is to prove the sequences $\psi_1,\dots,\psi_N$ remain close to each other when $k$ goes to $+\infty$. To this end, we introduce the following notation.  Given two sequences $x=(x^k)_{k\in\N}$ and $y=(y^k)_{k\in\N}$ of real numbers, we write
$$
x \sim y\ \mbox{ if }\ \limsup|x^k-y^k| < +\infty
$$
and
$$
x \preccurlyeq y\ \mbox{ if }\ \limsup(x^k-y^k) < +\infty.
$$
The following proposition is easy to check.
\begin{proposition}
	\begin{enumerate}
		\item For two sequences $x=(x^k)_{k\in\N}$ and $y=(y^k)_{k\in\N}$, $x\sim y$ if and only if $x\preccurlyeq y$ and $y\preccurlyeq x$.
		\item $\sim$ is an equivalence relation on $\R^\N$ and the relation $\preccurlyeq$ gives rise to a partial order on the quotient space.
	\end{enumerate}
\end{proposition}
In what follows, we denote the equivalence classes of $x,y\in\R^\N$ by $[x]$ and $[y]$ respectively, and we write $[x]\prec[y]$ if $[x]\preccurlyeq[y]$ and $[x]\not=[y]$.
\begin{remark}\label{rem-ordering_sequences}
	Although the equivalence classes are only partially ordered, proper subsequences can always be compared. Namely, given  $x=(x^k)_{k\in\N}$ and $y=(y^k)_{k\in\N}$, if neither $[x]\preccurlyeq[y]$ nor $[y]\preccurlyeq[x]$ holds then $\limsup(x^k-y^k)=\limsup(y^k-x^k)=+\infty$. Therefore, it is possible to find  subsequences for which (keeping the notation unchanged) $\lim(y^k-x^k)=+\infty$.  This clearly implies $[x]\preccurlyeq[y]$ for these subsequences.
\end{remark}
This comparison relation allows us to order the sequences $\psi_1,\dots,\psi_N$. More precisely, the following comparison holds:
\begin{lemma}
	Up to taking subsequences of $\psi_1,\dots,\psi_N$ and renumbering them, there exist integers $s_1,\dots,s_p$ such that $1\le s_1 < s_2 < \dots s_p < N$ and 
	$$
	[\psi_1] = \dots = [\psi_{s_1}]
	\prec [\psi_{s_1+1}] = \dots = [\psi_{s_2}] 
	\prec [\psi_{s_2+1}] \dots [\psi_{s_p}] \prec [\psi_{s_p+1}]
	= \dots = [\psi_N].
	$$
Besides, one can assume $(\psi_i^k-\psi_{i+1}^k)_{k\in\N}$ converge and
$$
\lim(\psi_i^k-\psi_{i+1}^k) =
\left\{\begin{array}{rl}
		-\infty & \mbox{if } i\in\{s_1,\dots,s_p\} \\
		b_i\in\R & \mbox{otherwise}
\end{array}\right..
$$
In particular,  $\lim(\psi_i^k-\psi_j^k)= -\infty$ whenever there exists $l$ such that $i\le s_l < j$ while $\lim(\psi_i^k-\psi_j^k)\in\R$ otherwise.	
	
\end{lemma}
\begin{proof}
	The proof is by finite induction on the number $N$ of sequences. Fix $i\in\{1,\dots,N-1\}$; assume 
	$$
	[\psi_1]\preccurlyeq[\psi_2]\preccurlyeq\dots\preccurlyeq[\psi_i]
	$$
	and, for each $j=1,\dots,i-1$, 
	$$\limsup(\psi_{j+1}^k-\psi_j^k)=\lim(\psi_{j+1}^k-\psi_j^k)\in\R\cup\{+\infty\}.$$
	
	 Let us now explain how to order the first $i+1$ sequences. The first step is to compare $[\psi_{i+1}]$ with $[\psi_i]$. There are three cases to consider:
	\begin{enumerate}
		\item if $[\psi_i]\preccurlyeq[\psi_{i+1}]$, extract  subsequences so that $\limsup(\psi_{i+1}^k-\psi_i^k)=\lim(\psi_{i+1}^k-\psi_i^k)\in\R\cup\{+\infty\}$.
		\item if $\limsup(\psi_{i+1}^k-\psi_i^k)=\limsup(\psi_i^k-\psi_{i+1}^k)=+\infty$, using Remark \ref{rem-ordering_sequences}, there exists a subsequence such that $[\psi_i]\preccurlyeq[\psi_{i+1}]$ and $\limsup(\psi_{i+1}^k-\psi_i^k)=\lim(\psi_{i+1}^k-\psi_i^k)\in\R\cup\{+\infty\}$.
		\item otherwise, $[\psi_{i+1}]\prec[\psi_i]$ and, up to extracting a subsequence, $\lim(\psi_{i+1}^k-\psi_i^k)=-\infty$.
	\end{enumerate}
	In the first two cases, we get $[\psi_1]\preccurlyeq[\psi_2]\preccurlyeq\dots\preccurlyeq[\psi_i]\preccurlyeq[\psi_{i+1}]$. In the last case, we then compare $[\psi_{i+1}]$ with $[\psi_{i-1}]$. After at most $i$ comparisons we get a total ordering of the first $i+1$  sequences which reads, after permuting the indices if necessary: $[\psi_1]\preccurlyeq[\psi_2]\preccurlyeq\dots\preccurlyeq[\psi_i]\preccurlyeq[\psi_{i+1}]$.
	
	The integers $s_1,\dots,s_p$ are the indices $i$ for which $[\psi_i]\prec[\psi_{i+1}]$.
\end{proof}
%

%
\subsubsection*{A modified maximizing sequence}
For each $k\in\N$, let us define the function $\varphi^k$ on $\Sm$  by
$$
\varphi^k(\eta)=\min_{i=1,\dots,N}\{c(\eta,\xi_i^k)-\psi_i^k\}.
$$
For $i=1,\dots,N$, let $V_i^k$ be  the ``weighted Vorono\"{\i} domains'' defined by
$$
V_i^k = \{\eta\in\Sm\ |\ \forall j=1,\dots,N\ \ c(\eta,\xi_i^k)-\psi_i^k\le c(\eta,\xi_j^k)-\psi_j^k \}.
$$
According to Lemma \ref{lem-discretization_mu},  the distance between an arbitrary point in  $\Sm$ and $\{\xi_1^k,\cdots,\xi_N^k\}$ is less than $\hpi$, thus $V_i^k\subset B(\xi_i^k,\hpi)$. Now, recall that $ \tilde{\varphi}^k$ and $\tilde{\psi}^k$ are $c$-conjugate while  $\psi_i^k= \max_{\bar{W}_i\cap \spt(\mu)} \tilde{\psi}^k$. Consequently,  $\tilde{\varphi}^k\le\varphi^k$ and $\varphi^k=c(.,\xi_i^k)-\psi_i^k$ on $V_i^k$.

\begin{lemma}\label{lem-using_Alexandrov}
	There exist $k_0\in\N$ and positive numbers $\varepsilon_1,\dots,\varepsilon_p$ such that for every $k\ge k_0$ and for every $ l\in\{1,\dots,p\}$, 
	$$
	 a_1+\dots+a_{s_l} \ge \sigma(V_1^k)+\dots+\sigma(V_{s_l}^k)+\varepsilon_l
	$$
\end{lemma}
\begin{proof}
	Fix $l\in\{1,\dots,p\}$. For any $k\in\N$ and any $i\in\{1,\dots,s_l\}$ we have, from the very definition of $V_i^k$,
	$$
	V_i^k \subset \{\eta\in B(\xi_i^k,\hpi)\ |\ \forall j=s_l+1,\dots,N\ \ \psi_j^k-\psi_i^k\le c(\eta,\xi_j^k) \}.
	$$
	Fix $\varepsilon>0$. Since $\lim(\psi_i^k-\psi_j^k)=-\infty$ for any $i,j$ such that $i\le s_l<j$, there exists $k_l\in\N$ such that
	$$
	\forall k\ge k_l\ \ \forall i\in\{1,\dots,s_l\}\ \ \forall j\in\{s_l+1,\dots,N\}\ \ \forall \eta\in V_i^k\ \ \ d(\eta,\xi_j^k)\ge\hpi-\varepsilon. 
	$$
	Therefore, for any $k\ge k_l$ and any $i,j$ such that $i\le s_l<j$, we obtain 
	$$V_i^k\cap B(\xi_j^k,\hpi-\varepsilon)=\emptyset.$$ 
	This gives $
	\forall k\ge k_l\ \ \ \ \  \Bigl(\bigcup_{i=1}^{s_l}V_i^k\Bigr)\cap\Bigl(\bigcup_{j=s_l+1}^NB(\xi_j^k,\hpi-\varepsilon)\Bigr) = \emptyset
	$. In other terms,
	\begin{equation}\label{eqn-using_Alexandrov}
		\forall k\ge k_l\ \ \ \ \ \bigcup_{i=1}^{s_l}V_i^k \subset \Sm\setminus\bigcup_{j=s_l+1}^NB(\xi_j^k,\hpi-\varepsilon).
	\end{equation}

	Let $\omega=\conv(\xi_{s_l+1}^k,\dots,\xi_N^k)\subset\Sm$ be the the convex hull of $\xi_{s_l+1}^k,\dots,\xi_N^k$. At this stage, there are two cases to consider depending on whether $\omega$  is the whole sphere or not.
	
	\textbf{First case:} $\conv(\xi_{s_l+1}^k,\dots,\xi_N^k)=\Sm$. This implies that the $\hpi$-neighborhood of $\{\xi_{s_l+1}^k,\dots,\xi_N^k\}$ is the whole sphere: $\Sm=\bigcup_{j=s_l+1}^NB(\xi_j^k,\hpi)$. Therefore, the above inclusion gives
	$$
	\forall k\ge k_l\ \ \ \ \ \bigcup_{i=1}^{s_l}V_i^k \subset \bigcup_{j=s_l+1}^NB(\xi_j^k,\hpi)\setminus B(\xi_j^k,\hpi-\varepsilon).
	$$
	Since $\sigma\bigl(B(\xi_j^k,\hpi)\setminus B(\xi_j^k,\hpi-\varepsilon)\bigr)$ does not depend on $\xi_i^k$, we get
	$$
	\forall k\ge k_l\ \ \ \ \ \sum_{i=1}^{s_l}\sigma(V_i^k) \le N\sigma\bigl(B(\xi,\hpi)\setminus B(\xi,\hpi-\varepsilon)\bigr),
	$$
	where $\xi \in \Sm$ is arbitrary. Since $a_1+\dots+a_{s_l}>0$, the following inequality is satisfied for $\varepsilon_l=\frac{1}{2}(a_1+\dots+a_{s_l})$ provided that $\varepsilon$ is small enough.

	$$
	\forall k\ge k_l\ \ \ \ \ \sum_{i=1}^{s_l}\sigma(V_i^k) \le a_1+\dots+a_{s_l} - \varepsilon_l
	$$
	and the lemma is proved in this case.

	\textbf{Second case:} $\conv(\xi_{s_l+1}^k,\dots,\xi_N^k)\not=\Sm$. Using that $\mu_k(\omega)\ge\sum_{j=s_l+1}^Na_j$, \eqref{eqn-Alexandrov_mu_k} gives
	\begin{equation}\label{eqn-2nd_case-1}
		\sum_{i=1}^{s_l}a_i \ge \mu_k(\Sm\setminus\omega) \ge \sigma(\omega^*) + \alpha.
	\end{equation}
	The definition of $\omega^*$ allows us to rewrite $\sigma(\omega^*)$ as 
	\begin{equation}\label{eqn-2nd_case-2}
		\sigma(\omega^*) = \sigma\Bigl(\Sm\setminus\bigcup_{j=s_l+1}^NB(\xi_j^k,\hpi)\Bigr) = \sigma(\Sm) - \sigma\Bigl(\bigcup_{j=s_l+1}^NB(\xi_j^k,\hpi)\Bigr).
	\end{equation}
	Writing $B(\xi_i^k,\hpi)=B(\xi_i^k,\hpi-\varepsilon)\cup \big(B(\xi_i^k,\hpi)\setminus B(\xi_i^k,\hpi-\varepsilon)\big)$, we get
	\begin{eqnarray}
			\sigma\Bigl(\bigcup_{j=s_l+1}^NB(\xi_j^k,\hpi)\Bigr) & \le & \sigma\Bigl(\bigcup_{j=s_l+1}^NB(\xi_j^k,\hpi-\varepsilon)\Bigr) \nonumber\\
			 & & + \sigma\Bigl(\bigcup_{j=s_l+1}^NB(\xi_i^k,\hpi)\setminus B(\xi_i^k,\hpi-\varepsilon)\Bigr) \nonumber\\
			 & \le & \sigma\Bigl(\bigcup_{j=s_l+1}^NB(\xi_j^k,\hpi-\varepsilon)\Bigr) + N\sigma\bigl(B(\xi,\hpi)\setminus B(\xi,\hpi-\varepsilon)\bigr). \nonumber
	\end{eqnarray}
	By combining this inequality together with \eqref{eqn-2nd_case-1} and \eqref{eqn-2nd_case-2}, we obtain
	$$
	\sum_{i=1}^{s_l}a_i \ge \sigma\Bigl(\Sm\setminus\bigcup_{j=s_l+1}^NB(\xi_j^k,\hpi-\varepsilon)\Bigr) - N\sigma\bigl(B(\xi,\hpi)\setminus B(\xi,\hpi-\varepsilon)\bigr) + \alpha.
	$$
	Choosing $\varepsilon$ small enough and using \eqref{eqn-using_Alexandrov}, there exists $k_l\in\N$ such that
	$$
	\forall k\ge k_l\ \ \ \ \ \sum_{i=1}^{s_l}a_i \ge \sum_{i=1}^{s_l}\sigma(V_i^k) + \frac{\alpha}{2}.
	$$
	
	Taking $k_0=\max(k_1,\dots,k_p)$ gives the result.
\end{proof}

%
\subsubsection*{The maximizing sequence admits a converging subsequence}
Since the elements of the maximizing sequence $(\tilde{\varphi}^k,\tilde{\psi}^k)_{k\in\N}$ are $c$-conjugate pairs, using Proposition \ref{prop-compacity} we are left with proving that the functions $\tilde{\psi}^k$ are uniformly bounded from below. We shall infer this bound from an upper bound on $\cK(\tilde{\varphi}^k,\tilde{\psi}^k)$. 

 First, using $\tilde{\varphi}^k\le\varphi^k$ and the definitions of $V_i^k$ and $\psi_i^k$, we get
\begin{eqnarray}\label{eqn-funct_upper_1}
	\cK(\tilde{\varphi}^k,\tilde{\psi}^k) & \le & \int_{\Sm}F(\varphi^k)d\sigma + \int_{\Sm}G(\tilde{\psi}^k)d\mu \nonumber\\
	 & \le & \sum_{i=1}^N \int_{V_i^k}F(c(\eta,\xi_i^k)-\psi_i^k)d\sigma(\eta) + \int_{W_i}G(\tilde{\psi}^k)d\mu \nonumber\\
	 & \le & \sum_{l=0}^p\sum_{i=s_l+1}^{s_{l+1}} \int_{V_i^k}F(c(\eta,\xi_i^k)-\psi_i^k)d\sigma(\eta) + a_iG(\psi_i^k)
\end{eqnarray}
where $s_0=0$ and $s_{p+1}=N$. According to Proposition \ref{prop-properties_F_G}, the functions $F$ and $G$ are sublinear, this gives for any $i$ and $k$,
\begin{eqnarray}\label{eqn-funct_upper_2}
	\int_{V_i^k}F(c(\eta,\xi_i^k)-\psi_i^k)d\sigma(\eta) + a_iG(\psi_i^k) & \le & \int_{V_i^k}c(\eta,\xi_i^k)d\sigma(\eta) + (a_i-\sigma(V_i^k))\psi_i^k \nonumber\\
	 & \le & C_1 + (a_i-\sigma(V_i^k))\psi_i^k
\end{eqnarray}
since $V_i^k\subset B(\xi_i^k,\hpi)$ and $c(.,\xi)\in L^1(B(\xi,\hpi),\sigma)$ for any $\xi$ (cf. Proposition \ref{prop-cost_L1}). Recall that for any $i\in\{s_l+1,\dots,s_{l+1}\}$, $\psi_i\sim\psi_{s_{l+1}}$ so that the sequence $(\psi_i^k-\psi_{s_{l+1}}^k)_{k\in\N}$ remains bounded, thus there exists a constant $A_i$ such that
\begin{eqnarray}
	(a_i-\sigma(V_i^k))\psi_i^k & = & (a_i-\sigma(V_i^k))\psi_{s_{l+1}}^k + (a_i-\sigma(V_i^k))(\psi_i^k-\psi_{s_{l+1}}^k) \nonumber\\
	 & \le & (a_i-\sigma(V_i^k))\psi_{s_{l+1}}^k + (\mu(\Sm)+\sigma(\Sm))A_i. \nonumber
\end{eqnarray}
Define $C_2=(\mu(\Sm)+\sigma(\Sm))\max\{A_1,\dots,A_N\}$ and insert the above inequality into \eqref{eqn-funct_upper_2}. Then, combining the resulting inequality with \eqref{eqn-funct_upper_1} yields
\begin{equation}\label{eqn-funct_upper_3}
	\cK(\tilde{\varphi}^k,\tilde{\psi}^k) \le N(C_1+C_2) + \sum_{l=0}^p \psi_{s_{l+1}}^k\sum_{i=s_l+1}^{s_{l+1}}(a_i-\sigma(V_i^k))
\end{equation}
This inequality combined with Lemma \ref{lem-using_Alexandrov}  is the key point to prove the lower bound on the sequences. Let us fix a constant $C_3$ such that for any $k$, $ C_3 \leq \cK(\tilde{\varphi}^k,\tilde{\psi}^k) -N(C_1+C_2)$. Since $\lim(\psi_{s_l}^k-\psi_{s_{l+1}}^k)=-\infty$ for $l=1,\dots,p$, we can assume, for any large $k$, $\psi_{s_1}^k \le \psi_{s_2}^k \le \dots \le \psi_{s_p}^k \le \psi_N^k$. Using these inequalities one after another together with Lemma \ref{lem-using_Alexandrov}, there exists a constant $C_3$ such that
\begin{eqnarray}\label{eqn-funct_upper_4}
	C_3 & \le & \psi_{s_1}^k\Bigl(a_1+\dots+a_{s_1}-\sigma(V_1^k)-\dots-\sigma(V_{s_1}^k)\Bigr) + \sum_{l=1}^p \psi_{s_{l+1}}^k\sum_{i=s_l+1}^{s_{l+1}}(a_i-\sigma(V_i^k)) \nonumber\\
	 & \le & \psi_{s_2}^k\Bigl(a_1+\dots+a_{s_2}-\sigma(V_1^k)-\dots-\sigma(V_{s_2}^k)\Bigr) + \sum_{l=2}^p \psi_{s_{l+1}}^k\sum_{i=s_l+1}^{s_{l+1}}(a_i-\sigma(V_i^k)) \nonumber\\
	 & \dots & \dots \nonumber\\
	 & \le & \psi_N^k\Bigl(a_1+\dots+a_N-\sigma(V_1^k)-\dots-\sigma(V_N^k)\Bigr)
\end{eqnarray}
In other terms, we have proved
$$
C_3 \le \psi_N^k(\mu(\Sm)-\sigma(\Sm)).
$$
Thus, $\mu(\Sm)>\sigma(\Sm)$ implies the sequence $\psi_N$ is bounded from below.

 Assume now that not all the sequences $\psi_1,\dots,\psi_N$ are equivalent. Then $p\ge 1$. Moreover,  writing again all the inequalities in \eqref{eqn-funct_upper_4} but the last one, then using Lemma \ref{lem-using_Alexandrov}, we obtain
\begin{eqnarray}
	C_3 & \le & \psi_{s_p}^k(a_1+\dots+a_{s_p}-\sigma(V_1^k)-\dots-\sigma(V_{s_p}^k)) \nonumber\\
	 & & + \psi_N^k(a_{s_p+1}+\dots+a_N-\sigma(V_{s_p+1}^k)-\dots-\sigma(V_N^k)) \nonumber\\
	 & \le & \psi_{s_p}^k\varepsilon_p + |\psi_N^k|(\mu(\Sm)+\sigma(\Sm)), \nonumber
\end{eqnarray}
which is a contradiction because $\psi_N$ is bounded and $\lim\psi_{s_p}^k=\lim(\psi_{s_p}^k-\psi_N^k)+\psi_N^k=-\infty$. Therefore, all the sequences are equivalent and uniformly bounded.

We are now in position to construct a converging subsequence. Fix $k\in\N$, let $\xi_0\in\Sm$ be a point where $\tilde{\psi}^k$ reaches its minimum. Let $i_0$ be such that $\xi_0\in W_{i_0}$. From Lemma \ref{lem-discretization_mu} we know that $d(\xi_0,\xi_{i_0}^k)\le\hpi-\beta$; combining this inequality with Proposition \ref{prop-min_c-concave}, we get
$$
\min_{\Sm}\tilde{\psi}^k =  \tilde{\psi}^k(\xi_0) \ge \tilde{\psi}^k(\xi_{i_0}^k) - C_4 = \psi_{i_0}^k -C_4 \ge C_5
$$
Since the functions $\tilde{\psi}^k$ are nonpositive, $(\tilde{\psi}^k)_{k\in\N}$ is uniformly bounded. Thus, Proposition \ref{prop-compacity} guarantees the maximizing sequence $(\tilde{\varphi}^k,\tilde{\psi}^k)_{k\in\N}$ admits a uniformly converging subsequence to some pair $(\tilde{\varphi},\tilde{\psi})\in\cA'$. We are left with proving $\lim\cK(\tilde{\varphi}^k,\tilde{\psi}^k)=\cK(\tilde{\varphi},\tilde{\psi})$. Since $G\le 0$ and the functions $\tilde{\psi}^k$ are uniformly bounded, we can apply the dominated convergence theorem and get $\lim\int_{\Sm}G(\tilde{\psi}^k)d\mu=\int_{\Sm}G(\tilde{\psi})d\mu$. On the other hand, the functions $F(\tilde{\varphi}^k)$ are uniformly bounded from above, thus Fatou's lemma implies $\lim\int_{\Sm}F(\tilde{\varphi}^k)d\sigma \le \int_{\Sm}F(\tilde{\varphi})d\sigma$. Since $(\tilde{\varphi},\tilde{\psi})\in\cA'$, equality must hold in this inequality and the proof is complete. 

Recall the functional does not decrease through the double convexification process, therefore we can assume  $(\tilde{\varphi},\tilde{\psi})$ to be a $c$-conjugate pair. Moreover, the function $\tilde{\varphi}$ being continuous on $\Sm$, it is bounded from above and $\cK(\tilde{\varphi},\tilde{\psi})<+\infty$. We have proved the problem $(NLK)$ is well-posed and admits a maximizing pair in $\cA'$.
%
\subsection{The components of an optimal pair do not vanish}
%
The maximizing pair obtained in the previous section could a priori vanish at some points. In order to apply Theorem \ref{thm-from_NLK_to_transport}, we must prove this maximizing pair is in $\cA$. To this end, we now prove that a $c$-conjugate pair vanishing at some point is not maximizing.

According to Proposition \ref{prop-min_c-concave} (2), it suffices to consider the case of a $c$-conjugate pair $(\varphi,\psi)$ such that $\psi(\xi_0)=0$ at some point $\xi_0\in\Sm$. We will construct a new pair $(\varphi^\varepsilon,\psi^\varepsilon)$ of admissible functions with $\cK(\varphi,\psi)<\cK(\varphi^\varepsilon,\psi^\varepsilon)$.

Since $\mu(\{\xi_0\})<\frac{1}{2}\sigma(\Sm)$, there exists $\beta>0$ such that $\mu(B(\xi_0,\beta))<\frac{1}{2}\sigma(\Sm)$. In what follows, we set $W_1$ the open ball defined by $W_1=B(\xi_0,\beta)$ and $W_2=\Sm\setminus W_1$ its complement. Since $\varphi=\psi^c$, we get
$$
\varphi(\eta) = \inf_{\xi\in\Sm}\bigl(c(\eta,\xi)-\psi(\xi)\bigr) = \inf\Bigl\{\inf_{\xi\in W_1}\bigl(c(\eta,\xi)-\psi(\xi)\bigr),\inf_{\xi\in W_2}\bigl(c(\eta,\xi)-\psi(\xi)\bigr)\Bigl\}.
$$
Let us set
$$
V_1=\Bigl\{\eta\in\Sm\ \Bigl|\ \inf_{\xi\in W_1}\bigl(c(\eta,\xi)-\psi(\xi)\bigr)<\inf_{\xi\in W_2}\bigl(c(\eta,\xi)-\psi(\xi)\bigr)\Bigr\},
$$
so that $\varphi(\eta)=\inf_{\xi\in W_1}\bigl(c(\eta,\xi)-\psi(\xi)\bigr)$ on $V_1$ and $\varphi(\eta)=\inf_{\xi\in W_2}\bigl(c(\eta,\xi)-\psi(\xi)\bigr)$ on $V_2=\Sm\setminus V_1$.

For $\varepsilon>0$, let $\psi^\varepsilon=\psi-\varepsilon\ind_{W_1}$ and $\varphi^\varepsilon=\inf_{\xi\in\Sm}(c(.,\xi)-\psi^\varepsilon(\xi))$ its $c$-conjugate function, in particular $(\varphi^\varepsilon,\psi^\varepsilon)$ is an admissible pair. As for $\varphi$, we define
$$
V_1^\varepsilon=\Bigl\{\eta\in\Sm\ \Bigl|\ \inf_{\xi\in W_1}\bigl(c(\eta,\xi)-\psi^\varepsilon(\xi)\bigr)<\inf_{\xi\in W_2}\bigl(c(\eta,\xi)-\psi^\varepsilon(\xi)\bigr)\Bigr\}
$$
and $V_2^\varepsilon=\Sm\setminus V_1^\varepsilon$.

The rest of the proof relies on the fact that $c(\eta,\xi)=\Lambda(d(\eta,\xi))$ where
$$
\Lambda(r) = \left\{\begin{array}{rl}
-\ln(\cos r) & \mbox{ if } r < \hpi \\
+\infty & \mbox{ otherwise.}
\end{array}\right.
$$
\begin{lemma}\label{lem-properties_V_1}
	The domains $V_1^\varepsilon$ satisfy the following properties:
	\begin{enumerate}
		\item if $\varepsilon_1\le\varepsilon_2$ then $V_1^{\varepsilon_2}\subset V_1^{\varepsilon_1}\subset V_1$,
		\item if $\varepsilon<\Lambda(\beta)$ then $\xi_0\in V_1^\varepsilon$,
		\item the $V_1^\varepsilon$ are open sets and $V_1=\bigcup_{\varepsilon<\Lambda(\beta)}V_1^\varepsilon$,
		\item $\varphi^\varepsilon=\varphi+\varepsilon$ on $V_1^\varepsilon$,
		\item $\varphi^\varepsilon=\varphi$ on $V_2$.
	\end{enumerate}
\end{lemma}
\begin{proof}
	Since $\psi^\varepsilon=\psi-\varepsilon$ on $W_1$ and $\psi^\varepsilon=\psi$ on $W_2$, we infer
	\begin{equation}\label{eqn-min_on_W_1}
			\inf_{\xi\in W_1}\bigl(c(\eta,\xi)-\psi^\varepsilon(\xi)\bigr) = \inf_{\xi\in W_1}\bigl(c(\eta,\xi)-\psi(\xi)\bigr)+\varepsilon
	\end{equation}
	and
	\begin{equation}\label{eqn-min_on_W_2}
			\inf_{\xi\in W_2}\bigl(c(\eta,\xi)-\psi^\varepsilon(\xi)\bigr) = \inf_{\xi\in W_2}\bigl(c(\eta,\xi)-\psi(\xi)\bigr),
	\end{equation}
	which easily implies (1).
	
	For $\varepsilon<\Lambda(\beta)$, the following holds 
	$$
	\inf_{\xi\in W_1}\bigl(c(\xi_0,\xi)-\psi^\varepsilon(\xi)\bigr) \le c(\xi_0,\xi_0)-\psi^\varepsilon(\xi_0) = \varepsilon < 	\inf_{\xi\in W_2}\bigl(c(\xi_0,\xi)-\psi^\varepsilon(\xi)\bigr)
	$$
	where the first inequality and the equality follow from $\xi_0\in W_1$, while the last
	inequality follows from the combination of  $c(\xi_0,\xi)\ge\Lambda(\beta)$ on $W_2$ and $\psi^\varepsilon \leq 0$. This implies $\xi_0\in V_1^\varepsilon$.

	 Let us prove (3). First note that $\psi^\varepsilon=\psi$ on $W_2$ yields
	$$
	V_1^\varepsilon=\Bigl\{\eta\in\Sm\ \Bigl|\ \inf_{\xi\in W_1}\bigl(c(\eta,\xi)-\psi^\varepsilon(\xi)\bigr)<\inf_{\xi\in W_2}\bigl(c(\eta,\xi)-\psi(\xi)\bigr)\Bigr\}.
	$$
	The map $\eta\mapsto\inf_{\xi\in W_2}\bigl(c(\eta,\xi)-\psi(\xi)\bigr)$ is (only) upper semicontinuous on $\Sm$, therefore bounded from above by a constant $A$. Set $\tilde{\psi}=\psi-(A+1)\ind_{W_1}$, use $\psi\le 0$ to infer, for any $\zeta\in W_1$,
	$$
	c(\eta,\zeta)-\tilde{\psi}(\zeta) \ge A+1 > \inf_{\xi\in W_2}\bigl(c(\eta,\xi)-\psi(\xi)\bigr) = \inf_{\xi\in W_2}\bigl(c(\eta,\xi)-\tilde{\psi}(\xi)\bigr).
	$$
	Consequently $\inf_{\xi\in W_2}\bigl(c(\eta,\xi)-\psi(\xi)\bigr)=\inf_{\xi\in\Sm}\bigl(c(\eta,\xi)-\tilde{\psi}(\xi)\bigr)$.  Finally, since $\tilde{\psi}$ is bounded, Proposition \ref{prop-continuity-c-transform} implies the map $\eta\mapsto\inf_{\xi\in W_2}\bigl(c(\eta,\xi)-\psi(\xi)\bigr)$ is continuous on $\Sm$. Since the map $\eta\mapsto\inf_{\xi\in W_1}\bigl(c(\eta,\xi)-\psi^\varepsilon(\xi)\bigr)$ is clearly upper semicontinuous, we have proved $V_1^\varepsilon$ is open.
	
	The equality $V_1=\bigcup_{\varepsilon<\Lambda(\beta)}V_1^\varepsilon$ follows easily from (1) and  $\psi^\varepsilon=\psi$ on $W_2$.

	For any $\eta\in V_1^\varepsilon\subset V_1$,  \eqref{eqn-min_on_W_1} yields
	$$
	\varphi(\eta)+\varepsilon = \inf_{\xi\in W_1}\bigl(c(\eta,\xi)-\psi(\xi)\bigr) + \varepsilon = \inf_{\xi\in W_1}\bigl(c(\eta,\xi)-\psi^\varepsilon(\xi)\bigr) = \varphi^\varepsilon(\eta).
	$$
	This proves (4). Similarly, for any $\eta\in V_2\subset V_2^\varepsilon$, \eqref{eqn-min_on_W_2} yields
	$$
	\varphi^\varepsilon(\eta) = \inf_{\xi\in W_2}\bigl(c(\eta,\xi)-\psi^\varepsilon(\xi)\bigr) = \inf_{\xi\in W_2}\bigl(c(\eta,\xi)-\psi(\xi)\bigr) = \varphi(\eta),
	$$
	which proves (5).
\end{proof}

Our goal is now to prove the existence of $\varepsilon>0$ such that $\cK(\varphi,\psi)<\cK(\varphi^\varepsilon,\psi^\varepsilon)$. To this end, we estimate the difference thanks to Lemma \ref{lem-properties_V_1}(5):
\begin{eqnarray}
	\cK(\varphi^\varepsilon,\psi^\varepsilon)-\cK(\varphi,\psi) & = & \int_{V_1}F(\varphi^\varepsilon)-F(\varphi)d\sigma + \int_{W_1}G(\psi^\varepsilon)-G(\psi)d\mu \nonumber\\
	 & \ge & \int_{V_1^\varepsilon}F(\varphi^\varepsilon)-F(\varphi)d\sigma + \int_{W_1}G(\psi^\varepsilon)-G(\psi)d\mu \nonumber
\end{eqnarray}
where we use $\varphi^\varepsilon\ge\varphi$, $F$ is non-decreasing, and $V_1^\varepsilon\subset V_1$. From Lemma \ref{lem-properties_V_1}(4) and the definition of $\psi^\varepsilon$, we get
\begin{eqnarray}
	\cK(\varphi^\varepsilon,\psi^\varepsilon)-\cK(\varphi,\psi) & \ge & \int_{V_1^\varepsilon}F(\varphi+\varepsilon)-F(\varphi)d\sigma + \int_{W_1}G(\psi-\varepsilon)-G(\psi)d\mu \nonumber\\
	 & \ge & \int_0^\varepsilon\int_{V_1^\varepsilon}f(\varphi+s)d\sigma ds - \int_0^\varepsilon\int_{W_1}g(\psi-s)d\mu ds. \nonumber
\end{eqnarray}
Since the $V_1^\varepsilon$ are open subsets containing $\xi_0$, Lemma \ref{lem-properties_V_1}(1) implies the existence of $\varepsilon_0>0$ and $\hpi>R>0$ such that, for any $\varepsilon<\varepsilon_0$, $B(\xi_0,R)\subset V_1^\varepsilon$. Combining this together with the fact that $f$ is non-negative yields
\begin{equation}\label{eqn-lower_bound_1}
	\cK(\varphi^\varepsilon,\psi^\varepsilon)-\cK(\varphi,\psi) \ge  \int_0^\varepsilon\int_{B(\xi_0,R)}f(\varphi+s)d\sigma ds - \int_0^\varepsilon\int_{W_1}g(\psi-s)d\mu ds.
\end{equation}
Lemma \ref{lem-properties_V_1} also guarantees $\xi_0\in V_1$, so we infer
$$
\varphi(\xi_0) = \inf_{\xi\in W_1}\bigl(c(\xi_0,\xi)-\psi(\xi)\bigr) \le c(\xi_0,\xi_0)-\psi(\xi_0)=0.
$$
Since $\varphi$ is nonnegative, we get $0=\varphi(\xi_0)=\min(\varphi)$. According to Proposition \ref{prop-min_c-concave}, $\varphi\le c(.,\xi_0)$ on $B(\xi_0,R)$, combining this together with $f$ nonincreasing yields
\begin{equation}\label{eqn-lower_bound_2}
	\forall \eta\in B(\xi_0,R)\ \ \ f(\varphi(\eta)+s) \ge f(c(\eta,\xi_0)+s).
\end{equation}
On the other hand, $g$ is nondecreasing and $\psi\le 0$ give us
\begin{eqnarray}\label{eqn-lower_bound_3}
	- \int_0^\varepsilon\int_{W_1}g(\psi-s)d\mu ds & \ge & - \int_0^\varepsilon\int_{W_1}g(-s)d\mu ds \nonumber\\
	 & \ge & \mu(W_1)G(-\varepsilon).
\end{eqnarray}
Inserting \eqref{eqn-lower_bound_2} and \eqref{eqn-lower_bound_3} into \eqref{eqn-lower_bound_1}, and then using Proposition \ref{prop-properties_F_G}, we get
\begin{eqnarray}
	\cK(\varphi^\varepsilon,\psi^\varepsilon)-\cK(\varphi,\psi) & \ge & \int_0^\varepsilon\int_{B(\xi_0,R)}f(c(\eta,\xi_0)+s)d\sigma(\eta) ds + \mu(W_1)G(-\varepsilon) \nonumber\\
	 & = & \Bigl(\frac{\sigma(\Sm)}{2}-\mu(W_1)\Bigr)\sqrt{2\varepsilon} + o(\sqrt{\varepsilon}). \nonumber
\end{eqnarray}
Since $\mu(W_1)<\frac{\sigma(\Sm)}{2}$, the inequality $\cK(\varphi^\varepsilon,\psi^\varepsilon)>\cK(\varphi,\psi)$ holds for $\varepsilon$ small enough and $(\varphi,\psi)$ is not a maximizing pair.

\begin{remark}
	It is not surprising that the vertex condition is the key point for proving that $\tilde{\psi}$ does not vanish. In fact, a point $\xi$ with mass $\frac{\sigma(\Sm)}{2}$ can occur for a convex domain with a point at infinity (such as an ideal triangle in the hyperbolic plane, for example). At such a point, the radial function is infinite, and therefore $\tilde{\psi}=\ln(\tanh r)$ vanishes. 
\end{remark}

\appendix

%
%
\section{The Cauchy-Crofton formula in $\dS$}\label{CCrofton}
%
%
The Cauchy-Crofton formulas are classical tools in integral geometry in which the volume, or curvature integrals, relative to a hypersurface is expressed in terms of an integral on the space of totally geodesic submanifolds of a given dimension with respect to its kinematic measure. Note that all the totally geodesics submanifolds considered in this appendix are assumed to be complete. The Cauchy-Crofton formulas are well-known in the context of Riemannian space forms  \cite{Santalo} and were extended to Lorentzian geometries in \cite{Solanes,Solanes-Teufel}. For our purpose, we only need to express the volume of an hypersurface in $\dS$ using an integral on the space of space-like geodesics.

%
\subsection{Kinematic measures} 
%
The space $\cL^s$ of space-like geodesics in $\dS$ has a natural smooth manifold structure and admits a measure $\dl$ (unique up to a multiplicative constant) which is invariant under the isometric actions on $\dS$. We refer to \cite[\S 2.3]{Solanes} or \cite[\S 2]{Solanes-Teufel} for the construction of $\cL^s$ and the measure $\dl$. For the convenience of the reader we also point out another approach to show the existence of $\dl$. The latter allows us to obtain results for space-like convex bodies without any regularity assumption, see Lemma \ref{l-prop-kine-bound}. The construction is based on the duality induced by the pseudo-Riemannian metric on $\R^{m+2}$. Indeed, recall there is a one-to-one correspondence between the (space-like) totally geodesic submanifolds $L$ of dimension $k$ (both in $\hyp$ and $\dS$) and the induced vector subspaces $\mathrm{Span}(L) \subset \R^{m+2}$, see Section \ref{sec-geometry_convex_bodies}. If $L\subset\dS$ is a space-like $k$-plane, then $\mathrm{Span}(L)^\bot$ contains time-like directions and $\mathrm{Span}(L)^\bot\cap\hyp$ is a $(m-k)$-plane of $\hyp$. As a consequence, the space of space-like geodesics of $\dS$ can be smoothly identified with the space of $(m-1)$-dimensional totally geodesic  submanifolds of $\hyp$; the existence of $\dl$ then follows from that of a suitable measure $d\tilde{\ell}_{m-1}$ on the aforementioned set of submanifolds of $\hyp$, namely $\dl = d\tilde{\ell}_{m-1}$. Indeed, for space forms, the existence of kinematic measure $d\tilde{\ell}_k$ on $k$-dimensional totally geodesic submanifolds $\tilde{\mathcal{L}}_k$ is by now classical \cite{Santalo}; let us recall the construction in the hyperbolic case. First, consider the space $\tilde{\cO}_{m+1-k}$ of $(m+1-k)$-dimensional totally geodesic submanifolds passing through the origin $o$. Identifying each element with its tangent space at $o$, we can see $\tilde{\cO}_{m+1-k}$ as the Grassmannian of $(m+1-k)$-plane in $\R^{m+1}$, and this space has a natural invariant measure $dg_{m+1-k}$. Similarly to what is done in the Euclidean case, a set $L \in \tilde{\mathcal{L}}_k$ is parameterized by means of $(M,p)$, where $M \in {\tilde{\mathcal{O}}}_{m+1-k}$ is the unique element of ${\tilde{\mathcal{O}}}_{m+1-k}$ orthogonal to $L$, and $p$ is the point defined by $\{p\}=L \cap M$. The integral of a continuous function $u:\tilde{\cL}_k\to\R$ with respect to the measure $d\tilde{\ell}_k$ can be expressed as follows
\begin{equation}\label{e-kine_dual}
\int_{\tilde{\cL}_k}u(L)d\tilde{\ell}_k(L) = \int_{\tilde{\cO}_{m+1-k}}\int_M u(M,p)d\nu_M(p)dg_{m+1-k}(M),
\end{equation} 
where $\nu_M$ is the canonical Riemannian measure on the $(m+1-k)$-dimensional hyperbolic space $M$. We refer to \cite[Section 17.3]{Santalo} for more details.  

Let us now illustrate the combination of the above parameterization with the duality of totally geodesic submanifolds with the case of a geodesic $ \gamma \in {\cL}^s$. We can parameterize such a $\gamma$
by
\begin{equation}\label{eq-DeSGeodPara}
	\gamma(s) = \cos (s) \,c'_{\xi_a}(h_a) +\sin (s)\, \xi_b,
\end{equation}
where $\xi_a, \xi_b \in \Sm$ are orthogonal and $h_a\in\R_+$. In that case, the parameters of $\gamma$ are $(M,p)\in\tilde{\cO}_2\times\hyp$, where $T_oM=\text{Span}(\xi_a,\xi_b) \subset T_o \hyp$ and $p= c_{\xi_a}(h_a) $. Indeed, to the geodesic $\gamma$ corresponds the $(m-1)$-dimensional submanifold $\gamma^{\perp}$ in $\hyp$. The point $c_{\xi_a}(h_a)$ clearly belongs to both $\gamma^{\perp}$ and $\text{Span} (\xi_a,\xi_b)$. Therefore, $p= c_{\xi_a}(h_a) $ because the intersection of the previous sets is reduced to a single point.

In the following lemma, we consider the $\varepsilon$-tubular neighborhood of a space-like smooth hypersurface $S$ of $\dS$, this set is defined as the subset of points in $\dS$ which can be connected to $S$ by a geodesic orthogonal to $S$ of length at most $\varepsilon$. Using the identification introduced earlier in this paragraph, we can now prove (Item (1) is originally proved in \cite{Solanes-Teufel}):

\begin{lemma}\label{l-prop-kine-bound}
	The measure $\dl$ on $\cL^s$ has the following properties:
	\begin{enumerate}
		\item For any $\varepsilon>0$ the set of space-like geodesics contained in the $\varepsilon$-tubular neighborhood of a given space-like totally geodesic hypersurface has positive finite $\dl$ measure.
		\item Let $\Omega^* \subset \dS_+$ be a space-like convex body. The set of space-like geodesics contained in a support hyperplane of $\Omega^*$ is $\dl$-negligible.
	\end{enumerate}
\end{lemma}
\begin{proof}
	The proofs are based on \eqref{e-kine_dual}. Let us call $V_1$ and $V_2$ the two sets of geodesics. Since $\dl$ is invariant by isometry, it suffices to prove the first result when the totally geodesic hypersurface is the equator $\Sm$. In that case, we can easily compute the measure of $V_1$. Indeed, using the same notation as above,  if $\gamma\in\cL^s$ is parameterized by $\gamma(s)=\cos(s)\,c'_{\xi_a}(h_a)+\sin(s)\,\xi_b$, then $\gamma \in V_1$ if and only if $h_a \leq \varepsilon$. Thus, the geodesics in $V_1$ are precisely those whose parameters $(M,p)\in\tilde{\cO}_2\times\hyp$ are such that $p\in B(o,\varepsilon)\subset M$. Therefore, the measure of $V_1$ satisfies
	$$
	\dl(V_1) = \int_{\tilde{\cO}_2}\nu_M(B(o,\varepsilon))dg_{m+1-k}(M) = 2\pi|\tilde{\cO}_2|\,(\cosh(\varepsilon)-1) > 0,
	$$
	where $|\tilde{\cO}_2|$ is the Riemannian volume of the Grassmannian. 

	The proof of the second item is similar. Let $\hat{H}$ be a support hyperplane to $\Omega^*$ at $c'_{\eta}(h)$. Consider $\gamma$ a geodesic contained in $\hat{H}$ and going through $c'_{\eta}(h)$. Using the duality of convex bodies, support hyperplanes, and points, we get 
	\begin{equation}\label{eChaineDualite}
		H \supset \gamma^\bot\cap \hyp \supset \{y\},
	\end{equation} 
	where $H$ is the support hyperplane to $\Omega$ orthogonal to $c'_{\eta}(h)$, and $y \in \partial \Omega$ is orthogonal to $\hat{H}$, see Section \ref{sec-duality}. The parameters $(M,p)\in\tilde{\cO}_2\times\hyp$ of $\gamma$ are such that $M\in\tilde{\cO}_2$ is the totally geodesic 2-plane orthogonal to  $\gamma^\bot\cap \hyp$ through $o$, and $\{p\} =M \cap \gamma^\bot$. Let $\mathrm{P}_M :\hyp \longrightarrow M$ denote the orthogonal projection onto $M$. Using  \eqref{eChaineDualite} and $\{p\} =M \cap \gamma^{\perp}$, we immediately get $p=\mathrm{P}_M(y)\in\mathrm{P}_M(\partial \Omega)$. The latter property can be improved by noticing that $\mathrm{P}_M (\Omega)$ is a convex subset of $M$. By \eqref{eChaineDualite} again, $\mathrm{P}_M( \Omega)$ is contained in a half-plane determined by  $H \cap M$. Because $ \{p\} =M \cap \gamma^{\perp} \subset M \cap H$, the set $M \cap H$ is a support line to $\mathrm{P}_M(\Omega)$ at $p$, namely  $p \in\partial(\mathrm{P}_M(\Omega))$. The same argument shows $\partial(\mathrm{P}_M(\Omega))= \mathrm{P}_M(\partial \Omega) $. Therefore, $\gamma \in V_2$ if and only if its coordinates $(M,p)$ are such that $p\in\partial(\mathrm{P}_M(\Omega))$.
	  
	   Consequently, the measure of $V_2$ satisfies
	$$
	\dl(V_2)= \int_{\tilde{\cO}_2}\nu_M(\partial(\mathrm{P}_M(\Omega)))dg_{m+1-k}(M) = \int_{\tilde{\cO}_2}0\,dg_{m+1-k}(M)=0.
	$$
	
\end{proof}

	\begin{remark}\label{rem-lader}
	We can extend  the result  in Lemma \ref{l-prop-kine-bound}(1) to a smaller set of space-like geodesics. Without loss of generality, we can assume the totally geodesic hypersurface is $\Sm$. Let  $U \subset \Sm$ be a nonempty open subset. We claim that the set $V$ of space-like geodesics contained in the $\varepsilon$-tubular neighborhood of $\Sm$ and intersecting $U$ has positive $\dl$ measure. Indeed, a computation similar to the one above gives $\dl(V)=2\pi|\tilde{\mathcal{U}}|\,(\cosh(\varepsilon)-1) > 0$, where $\tilde{\mathcal{U}}\subset\tilde{\cO}_2$ is the nonempty open set of 2-planes which intersect $U$. The compactness of $\tilde{\cO}_2$ certainly implies $|\tilde{\mathcal{U}}|>0$.
	\end{remark}

%
\subsection{Approximation by smooth convex bodies}\label{Appro}
%
Some of the results on general convex bodies will be inferred from known properties of smooth convex bodies by using the following approximation result.
\begin{proposition} \label{prop-approximation_convex}
	Let $\Omega\subset\hyp$ be a convex body with the point $o$ in its interior and $\Omega^*$ be its polar body. There exists a sequence $(\Omega_k)_{k\in\N}$ of hyperbolic convex bodies with $o$ in their interiors such that the following holds:
	\begin{enumerate}
		\item for all $k\in\N$, the boundary $\dom_k^*$ is smooth, and $\Omega_k$ and $\Omega_k^*$ are strictly convex.
		\item the sequences of radial and support functions $(r_k)_{k\in\N}$ and $(h_k)_{k\in\N}$ converge uniformly to the radial and support functions $r$ and $h$ of $\Omega$.
		\item the sequence $(\Omega_k^*)_{k\in\N}$ (resp. $(\Omega_k)_{k\in\N}$) converges to $\Omega^*$ (resp.  $\Omega$) w.r.t. Hausdorff topology.
		\item $\lim_{k \rightarrow +\infty} T_k(\eta)=T(\eta)$ for $\sigma$-a.e. $\eta\in\Sm$ and $\lim_{k \rightarrow +\infty} S_k(\xi)=S(\xi)$ for $\sigma$-a.e. $\xi\in\Sm$.
		\item the volumes of the boundaries converge: $\lim\limits_{k \rightarrow +\infty}|\dom^*_k|=|\dom^*|$.
		\item the curvature measures weakly converge: $\lim\limits_{k \rightarrow +\infty}\mu_k=\mu$.
	\end{enumerate}
\end{proposition}
\begin{proof}
	Recall that by using convex cones in $\R^{m+2}$ one gets a one-to-one correspondence between the hyperbolic convex bodies $\Omega$ (resp. space-like convex bodies $\Omega^* \subset \dS$) and the $(m+1)$-dimensional Euclidean convex bodies $\Omega_E$ contained in the open unit ball (resp. $\Omega^*_E$ containing the closed unit ball in their interior), see Section \ref{sec-geometry_convex_bodies}. Let us also recall that the radial and support functions are related through the formulas
	$$
	h_E=\tanh(h) \qquad\mbox{and}\qquad r_E= \tanh (r).
	$$
	Moreover, thanks to the above formula together with the following equivalences
	\begin{equation}\label{appen_B_3}
		\xi \in T(\eta) \Leftrightarrow \tanh (h(\eta)) = \tanh (r(\xi)) \langle \xi, \eta \rangle  \Leftrightarrow \eta \in S(\xi), 
	\end{equation}
	we infer that the mapping $T$ (resp. $S$) coincides with its Euclidean counterpart (see Proposition \ref{prop-S_T_equality_case}) and is single-valued whenever $h$ (resp. $r$) is differentiable, that is $\sigma$-almost everywhere. Therefore, the function $\argth$ being Lipschitz on all compact subsets of $(-1,1)$,  all the items of Proposition \ref{prop-approximation_convex} but the last two follow from results in Euclidean geometry that we now briefly recall using the standard reference \cite[Sections 1.6, 1.7, and 1.8]{Schneider}.
	
	Recall that a Euclidean convex body, with o in its interior, is strictly convex if and only if its support function is $C^1$. Moreover, the radial and support functions of a convex body and its polar body are related by the formulas $r_{E}^*=1/ h_E$ and $h_{E}^*= 1/r_E$. Consequently, the gauge function defined by $U_{\Omega_E^*}(x)=\inf\{t\ge 0\ |\ x\in t\Omega_E^*\}$, or equivalently as the inverse of the radial function of $\Omega_E^*$, is a convex function. For $\varepsilon>0$, consider the function $U_\varepsilon^*=U_{\Omega_E^*}*\rho_\varepsilon+\varepsilon\frac{|.|^2}{2}$, where $(\rho_\varepsilon)_{\varepsilon>0}$ are standard mollifiers, and $|.|$ denotes the Euclidean norm. The function $U_\varepsilon^*$ is smooth and strictly convex. So is the related convex body $\Omega_{\varepsilon, E}^*:=\{ U^*_\epsilon \le 1 \}$. Thus, its polar body $\Omega_{\varepsilon, E}$ has smooth support function hence is strictly convex as well. The radial function $r_{\varepsilon,E}^*$ of $\Omega_{\varepsilon,E}^*$ converges uniformly to $r_E^*$ as $\varepsilon$ goes to $0$. This implies the convergence of $\Omega_{\varepsilon,E}^*$ to $\Omega_E^*$ w.r.t. Hausdorff distance as $\varepsilon$ goes to $0$ \cite[Theorem 1.8.11]{Schneider}. As a result, we obtain the uniform convergence of $r_{\varepsilon,E}=1/ h_{\varepsilon,E}^*$ to $r_E$ \cite[Theorem 1.8.11]{Schneider} and the Hausdorff convergence of $\Omega_{\varepsilon,E}$ to $\Omega_E$. Finally, the fact that $h_{\varepsilon,E}$ and $r_{\varepsilon,E}$ are $C^1$ combined with \eqref{appen_B_3} yields the result stated in (4). The proof of the last two items follows from the previous properties together with the definition  of the area measure of a space-like convex body \eqref{e_def_AreaMea}. It can be proved that the measures on $\Sm$ as in \eqref{e_def_AreaMea} relative to $(\Omega_k^*)_{k \in \N}$ converge in total variation distance to the corresponding measure relative to $\Omega^*$. Total variation convergence implies the convergence of the total mass of the measures; this gives (5). Item (6) follows from the total variation convergence  and the pointwise convergence of $Q_k$ to $Q$, where $Q_k$ and $Q$ are as in \eqref{eqn-projection_equator}. More details are provided in the proof of Theorem \ref{thm-Crofton-conv}, where a slightly more general result is proved, see \eqref{eqn-sTV}.
\end{proof}

%
\subsection{Cauchy-Crofton formula for smooth and non-smooth convex bodies in $\dS_+$} \label{sec-CCrofton}
%
In this part, we report the Cauchy-Crofton formula for $C^2$ space-like hypersurfaces in the de Sitter space proved by G. Solanes and E. Teufel \cite[Theorem~1]{Solanes-Teufel}. Precisely, we only use the formula involving lines (corresponding to $r=1$ in their statement). We also restrict our attention to hypersurfaces that are graphs over an open subset $\omega$ of the equator $\Sm$. This assumption guarantees that the homotopy condition required in the Solanes-Teufel result is satisfied. In the case where $\omega\neq \Sm$, we add some assumptions to guarantee that the boundary is fixed as required in \cite[Theorem 1]{Solanes-Teufel}.

We then use the approximation result proved in the previous paragraph to generalize the Cauchy-Crofton formula to boundaries of arbitrary space-like convex bodies in $\dS_+$.
\begin{theorem}[Solanes-Teufel \cite{Solanes-Teufel}]\label{thm-ST} 
	Let $\Sigma_1,\Sigma_2\subset\dS$ be two space-like  hypersurfaces of class $C^2$ possibly with boundary. Let us assume that, for $i\in \{1,2\}$,  
	$$
	\Sigma_i= \{( h_i(\eta),\eta)|\, \eta \in \omega\},
	$$ 
	where $\omega=\Sm$ or $\omega\subset\Sm$ is a domain with $C^2$ boundary. If $\omega \subsetneq \Sm$, further assume that $h_1=h_2$ on $\partial \omega$. Then, the following equality holds
	$$
	|\Sigma_2|-|\Sigma_1|  =  \frac{m}{|\Sbb^{m-1}|}\int_{\cL^s}(\#(\gamma\cap\Sigma_1)-\#(\gamma\cap\Sigma_2))\dl(\gamma).
	$$
\end{theorem}

\begin{remark}
	The space $\cL^s$ is non-compact but the integral is well-defined: given a $C^2$ surface $\Sigma$, any space-like geodesic close enough to the light cone intersect $\Sigma$ in exactly two points so that the function $\#(\gamma\cap\Sigma_1(\omega))-\#(\gamma\cap\Sigma_2(\omega))$ has compact support in $\cL^s$ \cite[Lemma 3]{Solanes-Teufel}.
	
\end{remark}

We now generalize the Cauchy-Crofton formula to boundaries of space-like convex bodies in the following way. 
\begin{theorem} \label{thm-Crofton-conv}
	Let $\Omega_1^*,\Omega_2^*\subset\dS_+$ be two space-like convex bodies, $h_1,h_2$ be their radial functions, and $\omega$ be either $\Sm$ or $\{\eta \in \Sm|\, h_1(\eta)<h_2(\eta)\}$. For $i\in \{1,2\}$, let $\Sigma_i\subset\dS$ be the hypersurface defined by 
	$\Sigma_i = \{( h_i(\eta),\eta)|\, \eta \in \omega\}.$ 
	
	Then, the following equality holds
	$$
	|\Sigma_2|-|\Sigma_1|  =  \frac{m}{|\Sbb^{m-1}|}\int_{\cL^s}(\#(\gamma\cap\Sigma_1)-\#(\gamma\cap\Sigma_2))\dl(\gamma).
	$$
	
	In particular, the Cauchy-Crofton formula holds for boundaries of space-like convex bodies in $\dS_+$. Besides, for $\omega = \{\eta \in \Sm|\, h_1(\eta)<h_2(\eta)\}$ and $\dl$-a.e. $\gamma\in\cL^s$, it holds
	$$
	\#(\gamma\cap\Sigma_1)-\#(\gamma\cap\Sigma_2) \geq 0, 
	$$
	and provided $\{\eta \in \Sm|\, h_1(\eta)<h_2(\eta)\} \neq \emptyset$, we actually have
	$$\frac{m}{|\Sbb^{m-1}|}\int_{\cL^s}(\#(\gamma\cap\Sigma_1)-\#(\gamma\cap\Sigma_2))\dl(\gamma)>0.$$
\end{theorem}
\begin{proof}
	For $i\in \{1,2\}$, let $(\Omega^*_{i,k})_{k\in\N}$ be the sequence of smooth strictly convex bodies in $\dS_+$ given by Proposition \ref{prop-approximation_convex} when applied to $\Omega_i=\Omega_i^{**}$. Writing $(r_i^{*,k}= h_i^k)_{k \in \N}$ 
	the associated sequences of radial
	 functions,	the space-like convex bodies $\Omega_{i,k}^*\subset\dS$ have smooth boundaries given by $\dom_{i,k}^*=\{(h_i^k(\eta),\eta)\ |\ \eta\in\Sm\}$.
	
	Let $\omega$ be either $\Sm$ or $\{\eta \in \Sm|\, h_1(\eta)<h_2(\eta)\}$, and define
	$$
	\Sigma_i= \{( h_i(\eta),\eta)|\, \eta \in \omega\} \qquad\mbox{and}\qquad
	\Sigma_{i,k}= \{( h_i^k(\eta),\eta)|\, \eta \in \omega\}.
	$$
	If $\omega=\Sm$ then $\Sigma_{i,k}=\dom_{i,k}^*$, and   Proposition~\ref{prop-approximation_convex} (5) guarantees that 
	$$
	\lim_{k\rightarrow+\infty}|\Sigma_{i,k}|=|\dom_i^*|=|\Sigma_i|.
	$$
	If $\emptyset\neq\omega\subsetneq\Sm$, let us prove the same volume convergence holds when the hypersurfaces are restricted to $\omega$. To this aim, let $(a_k)_{k\in\N}$ be a sequence of positive numbers, regular values of  $ \tanh h_2^k-\tanh h_1^k$, and converging to zero. We further assume that $a_k \geq 5\max\{ ||h_1 -h_1^k||_{\infty},||h_2 -h_2^k||_{\infty} \}$ and define
	$$
	\omega_k = \{\eta \in \Sm|\, \tanh h_1^k(\eta)-\tanh h_2^k(\eta) < -a_k\}.
	$$ 
	 
	Let us also define $\bar{h}_1^k= \argth (\tanh h_1^k +a_k)$ and $\bar{h}_2^k= h_2^k$. Note that $\tanh\bar{h}_1^k = \tanh h_1^k +a_k$ is the support function of a Euclidean convex body (namely the $a_k$-neighborhood of the Euclidean convex body whose support function is $\tanh h_1^k$, cf. \cite{Schneider}). Let $T_i^k$ denote the mapping relative to $\Omega_{i,k}=\Omega_{i,k}^{**} $ as in \eqref{eqn-definition_T}. In order to approximate $\Omega_i^*$ and $\Sigma_i$, we shall make use of $\bar{\Omega}^* _{i,k}$, the space-like convex body with radial function $\bar{h}_i^k$, 
	 and
	$$
	\bar{\Sigma}_{i,k}=\{(\bar{h}^k_i(\eta), \eta)|\ \eta \in \omega_k\} \subset \partial\bar{\Omega}_{i,k}^*.
	$$
	Given $\eta \in \omega_k$, note that $T_1^k(\eta)$ is not, in general, the set of normal vectors to $\partial\bar{\Omega}_{1,k}^*$ at $(\bar{h}^k_1(\eta), \eta)$.  	
	
	Let $\bar{r}_i^k$ denote the radial function associated to  $\bar{h}^k_i$, and $r_i$ denote the one associated to $h_i$. By construction, Proposition~\ref{prop-approximation_convex} implies the uniform convergences
	\begin{equation}\label{eq-UnifConv}
		\lim_{ k \rightarrow +\infty} \bar{h}^k_i = h_i \qquad\text{ and }\qquad \lim_{ k \rightarrow +\infty} \bar{r}_i^k = r_i,
	\end{equation}
	while, for $\sigma$-a.e. $\eta\in\Sm$,
	$$
	\lim_{ k \rightarrow +\infty} T_i^k(\eta) = T_i(\eta).
	$$
	
	By assumption on the $(\Omega_i^*)$'s and the above statements, there exists $c>0$ such that for all large $k$
	\begin{equation}\label{eq-UnifBound}
		1/c \leq h_1, h_2,r_1,r_2 \leq c  \qquad\mbox{ and }\qquad   1/c \leq \bar{h}_1^k, \bar{h}_2^k,\bar{r}_1^k,\bar{r}_2^k \leq c.
	\end{equation}

	Our goal is to show $\lim_{k \rightarrow + \infty}|\bar{\Sigma}_{i,k}|= |\Sigma_i|$. To achieve this aim, it is more convenient to use measures on $\Sm$ as in \eqref{e_def_AreaMea} rather than the area measures. Thus, we set
	$$
	\tilde{\sigma}_i = \frac{\cosh^{m+1}  h_i}{\cosh r_i \circ T_i } \sigma \qquad\text{  and  }\qquad \tilde{\sigma}_i^k=\frac{\cosh^{m+1} \bar{h}_i^k }{\cosh{\bar{r}}_i^k \circ T_i^k} \sigma.
	$$
	
	The estimates on the radial and support functions recalled above combined with the dominated convergence theorem, yield the total variation convergence of the measures $\tilde{\sigma}_i^k$ to $\tilde{\sigma}_i$, namely 
	\begin{equation}\label{eqn-sTV}
		\lim_{k \rightarrow +\infty} \sup_{f}\left| \int f\, \tilde{\sigma}_i^k - \int f\,\tilde{\sigma}_i\right| =0,
	\end{equation}
	where the supremum is taken over the set of all Borel functions $f : \Sm \rightarrow [-1,1] $.
	
	From the above estimate, we obtain, for any $\varepsilon >0$ and any sufficiently large $k$, 
	$$
	|\tilde{\sigma}_i^k(\omega_k) -\tilde{\sigma}_i(\omega_k)| \leq \varepsilon.
	$$
	To conclude, we infer from combining $h_i^k\rightarrow h_i$ together with the properties of $(a_k)_{k \in \N}$ that 
	$$
	\omega = \liminf \omega_k.
	$$
	Thus, we get
	$$
	\lim_{k \rightarrow + \infty}   \tilde{\sigma}_i (\omega_k) =  \tilde{\sigma}_i (\omega).
	$$
	Therefore, Theorem \ref{thm-ST} gives us
	\begin{equation}\label{eqn-proof-uni}
		|\bar{\Sigma}_{2,k}|-|\bar{\Sigma}_{1,k}|  =  \frac{m}{|\Sbb^{m-1}|}\int_{\cL^s}(\#(\gamma \cap \bar{\Sigma}_{1,k})-\#(\gamma\cap \bar{\Sigma}_{2,k}))\dl(\gamma)
	\end{equation}
	while our previous argument implies
	\begin{multline*}
		\lim_{k \rightarrow +\infty} |\bar{\Sigma}_{2,k}|-|\bar{\Sigma}_{1,k}| = \lim_{k \rightarrow +\infty} | \tilde{\sigma}_2 (\omega_k)|-| \tilde{\sigma}_1 (\omega_k)|  \\
		= | \tilde{\sigma}_2 (\omega)|-| \tilde{\sigma}_1 (\omega)|=|\Sigma_2|-|\Sigma_1|.
	\end{multline*}

	We are left with proving the convergence of the right-hand side of \eqref{eqn-proof-uni}  to the same expression with $\Sigma_i$ instead of $\bar{\Sigma}_{i,k}$. The same convergence needs to be proved when $\omega =\Sm$ and the integrand involves $\Sigma_{i,k}$ instead of $\bar{\Sigma}_{i,k}$.
	
	According to Lemma \ref{l-prop-kine-bound}, the set of geodesics in $\cL^s$ lying in an arbitrary support hyperplane to a space-like convex body in $\dS_+$ is negligible. In what follows, we discard the negligible set of geodesic lines lying in a support hyperplane to one of the following space-like convex bodies: $\Omega_1^*$, $\Omega_2^*$, and $\Omega_1^* \cap \Omega_2^*$.
	
	Let us first assume $\omega = \{\eta \in \Sm|\, h_1(\eta)<h_2(\eta)\}$. The proof is based on the fact that the value of $\#(\gamma\cap\Sigma_1)-\#(\gamma\cap\Sigma_2)$ only depends on topological properties relative to $\Omega_1^*$ and $\Omega_2^*$.  Precisely, up to the negligible set of geodesics introduced above, only the following configurations can occur:
	\begin{itemize}
		\item if $\gamma \cap (\Omega_1^* \cup \Omega_2^*)= \emptyset$ then $\#(\gamma\cap\Sigma_1)=\#(\gamma\cap\Sigma_2)=0$,
		\item if $\gamma \cap (\mathring{\Omega}_1^* \cup \mathring{\Omega}_2^*) \neq \emptyset$ then
		\begin{itemize}
			\item either  $\gamma \cap \mathring{\Omega}_2^* \neq \emptyset$ and \\
				$\#(\gamma\cap\Sigma_1)-\#(\gamma\cap\Sigma_2)=0$ (each intersection number belongs to $\{0,1,2\}$)
			\item or  $\gamma \cap {\Omega}_2^* = \emptyset$ (thus  $\gamma \cap \mathring{\Omega}_1^* \neq \emptyset$) and \\
			$\#(\gamma\cap\Sigma_1)=2$, $\#(\gamma\cap\Sigma_2)=0$.
		\end{itemize}
	\end{itemize}	 
	In particular, we infer that  $\#(\gamma\cap\Sigma_1)-\#(\gamma\cap\Sigma_2) \geq 0$ for $\dl$-a.e. $\gamma$.
	
	Note the above discussion is valid for any pair of space-like convex bodies and that it might be easier to convince oneself of the above case-by-case study by using the Euclidean counterparts of our space-like convex bodies as described in Section~\ref{sec-geometry_convex_bodies}. Moreover, note also that, for a given $\gamma$, each of the four conditions is open with respect to Hausdorff topology. Therefore, combining this together with the Hausdorff convergence of $\bar{\Omega}_{i,k}^*$ to $\Omega_i^*$, we get, for $\dl$-a.e. geodesic $\gamma$ and for $k$ sufficiently large\footnote{We also discard the corresponding negligible sets of geodesics relative to each $\bar{\Omega}_{1,k}^*$, $\bar{\Omega}_{2,k}^*$, and $\bar{\Omega}_{1,k}^*\cap\bar{\Omega}_{2,k}^*$.}, 
	$$
	\#(\gamma\cap\bar{\Sigma}_{1,k})-\#(\gamma\cap\bar{\Sigma}_{2,k})= \#(\gamma\cap\Sigma_1)-\#(\gamma\cap\Sigma_2).
	$$
	
	To complete the proof, we produce a set of geodesics of finite $\dl$-measure on which all $f_k(\gamma):=\#(\gamma\cap\bar{\Sigma}_{1,k})-\#(\gamma\cap\bar{\Sigma}_{2,k})$ are concentrated. By the above discussion, the function	$f_k$ vanishes whenever $\gamma$ meets the interior of $\bar{\Omega}_{1,k}^*\cap\bar{\Omega}_{2,k}^*$. Therefore, according to \eqref{eq-UnifBound}, $f_k(\gamma)=0$ whenever $\gamma(s) =  \cos (s) \,c'_{\xi_a}(h_a) +\sin (s)\, \xi_b $ with $h_a > c$. In other terms, the $c$-neighborhood of the equator $\Sm$ (where $c$ is as in \eqref{eq-UnifBound}) satisfies the required properties as proved in Lemma \ref{l-prop-kine-bound} (1).

	The case $\omega =\Sm$ can be proved along the same lines and is even simpler. The details are left to the reader.

	\vspace*{0.2cm}

	It remains to prove the last inequality assuming 
	$$
	\omega=\{\eta \in \Sm|\, h_1(\eta)<h_2(\eta)\} \neq \emptyset.
	$$
	Let $\eta_0\in\omega$ and $a>0$. Consider the space-like convex body $\tilde{\Omega}^*$ whose support function $\tilde{h}^*$ is defined by $\coth(\tilde{h}^*)=\coth(r_2) + {a}$, that is, the convex body in $\dS_+$ whose Euclidean counterpart in $\{1\}\times \R^{m+1}$ is the ${a}$-neighborhood of $\Omega^*_{2,E}$. Let $\tilde{h}$ denote the radial function of $\tilde{\Omega}^*$.
	
	For $a>0$ sufficiently small, the convex $\tilde{\Omega}^*$ contains $\Omega^*_2$ in its interior, and $c'_{\eta_0}(\tilde{h}(\eta_0))\in\mathring{\Omega}^*_1\setminus\Omega^*_2$. Therefore, given $H^*$ a support hyperplane to $\tilde{\Omega}^*$ at $c'_{\eta_0}(\tilde{h}(\eta_0))$,  any geodesic $\gamma$ contained in $H^*$ and intersecting $\mathring{\Omega}_1$ satisfies $\#(\gamma\cap\Sigma_1)=2$ and $\#(\gamma\cap\Sigma_2)=0$.
	
	By a  compactness argument, there exists $0<\varepsilon$ small such that the $\varepsilon$-neighborhood of $H^*$ does not intersect $\Omega^*_2$, and for any geodesic $\gamma$ contained in this $\varepsilon$-neighborhood and intersecting $\mathring{\Omega}_1 \cap H^*$, the equalities $\#(\gamma\cap\Sigma_1)=2$ and $\#(\gamma\cap\Sigma_2)=0$ still hold true. Since this set of geodesics has positive $\dl$-measure (cf. Lemma \ref{l-prop-kine-bound} and Remark \ref{rem-lader}), we infer
	$$
	\int_{\cL^s}(\#(\gamma\cap\partial\Sigma_1)-\#(\gamma\cap\partial\Sigma_2))\dl(\gamma)>0.
	$$
	\end{proof}

%
%
\section{Properties of $c$-concave functions}\label{app-analysis}
%
%
In this appendix we gather the main properties of $c$-concave functions on $\Sm$, where the cost function $c:\Sm\times\Sm\to\R\cup\{+\infty\}$ is given by
$$
c(\eta,\xi)=\left\{\begin{array}{rl}
-\ln(\langle\eta,\xi\rangle) & \mbox{ if } \langle\eta,\xi\rangle > 0 \\
+\infty & \mbox{ otherwise.}
\end{array}\right..
$$
Note that the cost function only depends on the geodesic distance in $\Sm$ between $\eta$ and $\xi$, as $c(\eta,\xi)=\Lambda(d(\eta,\xi))$, where
$$
\Lambda(r) = \left\{\begin{array}{rl}
-\ln(\cos r) & \mbox{ if } r < \hpi \\
+\infty & \mbox{ otherwise}
\end{array}\right..
$$
In the following we will use that $\Lambda$ is convex on $[0,\hpi)$ and $\lim_{r\to\hpi}\Lambda(r)=+\infty$.

Most of these properties of $c$-concave functions are now classical, at least when the cost function is real-valued. The main references in the real-valued case are \cite[\S 2.4]{Villani-1} or \cite[Chapter 5]{Villani-2}. To treat our particular case, we adapt arguments from \cite{Bertrand-2}.

\begin{definition} \label{def-c-concavity}
	Let $\psi:\Sm\to\R\cup\{-\infty\}$ be a function which is not identically $-\infty$.
	\begin{enumerate}
		\item The function $\psi$ is $c$-concave if there exists $\varphi:\Sm\to\R\cup\{-\infty\}$ such that
		$$
		\psi(\cdot) = \inf_{\eta\in\Sm}\{c(\eta,\cdot)-\varphi(\eta)\}.
		$$
		\item The $c$-transform of $\psi$ is the function $\psi^c:\Sm\to\R\cup\{-\infty\}$ defined by
		$$
		\psi^c(\cdot) = \inf_{\xi\in\Sm}\{c(\cdot,\xi)-\psi(\xi)\}.
		$$
		\item If $\psi$ is $c$-concave, its $c$-superdifferential is
		$$ \partial^c\psi(\xi) = \bigl\{\eta\in\Sm\ \bigl|\ \forall\zeta\in\Sm\ \ c(\eta,\xi) - \psi(\xi) \le c(\eta,\zeta)-\psi(\zeta) \bigr\}
		$$
		
		\item A pair of functions $(\varphi,\psi)$ is a $c$-conjugate pair if $\varphi=\psi^c$ and $\psi=\varphi^c$. 
	\end{enumerate}
\end{definition}
\begin{remark} \label{rem-conjugate_pairs}
	It is well-known that a function $\psi$ is $c$-concave if  and only if the image of $\psi^c$ is contained in $\R\cup\{-\infty\}$ and $\psi^{cc}=\psi$ (where $\psi^{cc}$ is a notation for $(\psi^c)^c$). In particular, the $c$-conjugate pairs are exactly the pairs $(\psi^c,\psi)$ where $\psi$ $c$-concave.
	
	If $(\varphi,\psi)$ is a $c$-conjugate pair then the $c$-superdifferential of $\psi$ satisfies
	$$ \partial^c\psi(\xi)=\{\eta\in\Sm\ |\ \psi^c(\eta)+\psi(\xi)=c(\eta,\xi)\},$$
	 so that $\partial^c\psi:\Sm\rightrightarrows\Sm$ is a multivalued map. Defining similarly $\partial^c \varphi$, we observe that the superdifferentials of $\psi$ and $\varphi$ are inverse of each other, namely
	$$
	\eta\in\partial^c\psi(\xi) \Leftrightarrow \xi\in\partial^c\varphi(\eta).
	$$
\end{remark}

The following proposition originates from  \cite[Proposition 4.4]{Bertrand-2}.
\begin{proposition} \label{prop-continuity-c-transform}
	Let $\psi:\Sm\to\R\cup\{-\infty\}$ be a function bounded from above such that
	\begin{equation} \label{eqn-condition_on_psi}
	\forall\eta\in\Sm\ \ B(\eta,\hpi)\cap\{\psi>-\infty\}\not=\emptyset.
	\end{equation}
	Then $\psi^c$ is real-valued and Lipschitz regular on $\Sm$, moreover its Lipschitz constant only depends on upper bounds of $\psi$ and $\psi^c$.
\end{proposition}
\begin{proof}
	Fix some $\eta\in\Sm$. According to \eqref{eqn-condition_on_psi}, there exists $\xi\in\Sm$ such that $c(\eta,\xi)<+\infty$ and $-\psi(\xi)<+\infty$. Therefore, we have $\psi^c(\eta)<+\infty$. On the other hand, if $A$ is an upper bound of $\psi$, we have $c(\eta,\xi)-\psi(\xi)\ge-A$ for any $\xi\in\Sm$; this implies $\psi^c(\eta)\ge-A$. Therefore the function $\psi^c$ is real-valued.
	
	Being an infimum of continuous functions, $\psi^c$ is upper semi-continuous. Since $\Sm$ is compact, $\psi^c$ is bounded from above.
	
	Let $A$ and $B$ be upper bounds of $\psi$ and $\psi^c$ respectively. In order to prove that $\psi^c$ is Lipschitz with a constant depending only on $A$ and $B$, we begin by proving it locally. Since $\psi^c+A\ge 0$, $A+B+1>0$ and $\Lambda^{-1}(A+B+1)$ is well-defined in $(0,\hpi)$. Let $\alpha=\hpi-\Lambda^{-1}(A+B+1)$ and fix some $\eta_0\in\Sm$. For any $\eta\in B(\eta_0,\frac{\alpha}{2})$ and any $\xi$ such that $d(\eta_0,\xi)\ge\hpi-\frac{\alpha}{2}$, we have $d(\eta,\xi)\ge\hpi-\alpha$, thus $c(\eta,\xi)\ge A+B+1\ge \psi(\xi)+\psi^c(\eta)+1$. Therefore $\psi^c(\eta)+1\le c(\eta,\xi)-\psi(\xi)$, and, by definition of $\psi^c$, we get
	$$
	\psi^c(\eta)=\inf\bigl\{c(\eta,\xi)-\psi(\xi)\ \bigl|\ \xi\in B(\eta_0,\hpi-\frac{\alpha}{2}) \bigr\}
	$$
	for any $\eta\in B(\eta_0,\frac{\alpha}{2})$. For any $\xi\in B(\eta_0,\hpi-\frac{\alpha}{2})$, the map $\eta\mapsto c(\eta,\xi)-\psi(\xi)$ is $L_{A,B}$-Lipschitz on $B(\eta_0,\frac{\alpha}{4})$, where $L_{A,B}=\Lambda'(\hpi-\frac{\alpha}{4})$. As an infimum of $L_{A,B}$-Lipschitz functions, $\psi^c$ is $L_{A,B}$-Lipschitz on $B(\eta_0,\frac{\alpha}{4})$. 
	
	For any $\eta,\eta'$ in $\Sm$, consider a finite family of points $\eta=\eta_0,\dots,\eta_k,\dots,\eta_N=\eta'$ on a minimizing geodesic between $\eta$ and $\eta'$ such that $d(\eta_k,\eta_{k+1})<\frac{\alpha}{2}$. We get
	\begin{eqnarray}
	|\psi^c(\eta)-\psi^c(\eta')| & \le & \sum_{k=0}^{N-1}|\psi^c(\eta_k)-\psi^c(\eta_{k+1})| \nonumber \\
	& \le & L_{A,B} \sum_{k=0}^{N-1}d(\eta_k,\eta_{k+1}) \nonumber \\
	& \le & L_{A,B}d(\eta,\eta'), \nonumber
	\end{eqnarray}
	and $\psi^c$ is $L_{A,B}$-Lipschitz.
\end{proof}

\begin{remark} \label{rem-double_convexification}
	Starting from a pair $(\varphi,\psi)$ of functions such that $\psi:\Sm\to\R\cup\{-\infty\}$, \eqref{eqn-condition_on_psi} holds, and  $\forall \eta,\xi\ \ \varphi(\eta)+\psi(\xi)\le c(\eta,\xi)$, the double convexification trick gives a $c$-conjugate pair of functions which are greater or equal to $\varphi$ and $\psi$ respectively.
	
	Indeed, for any $\eta,\xi$,  $\varphi(\eta)\le c(\eta,\xi)-\psi(\xi)$ implies, by taking the infimum on $\xi$, $\forall \eta\ \ \varphi(\eta)\le\psi^c(\eta)$. Similarly,  $\psi(\xi)\le c(\eta,\xi)-\psi^c(\eta)$ yields $\psi(\xi)\le\psi^{cc}(\xi)$. Then, we infer from Remark \ref{rem-conjugate_pairs} that $(\psi^c,\psi^{cc})$ is a $c$-conjugate pair with $\varphi\le\psi^c$ and $\psi\le\psi^{cc}$.
\end{remark}

\begin{proposition}\label{prop-min_c-concave}
	Any $c$-conjugate pair $(\varphi,\psi)$ satisfies:
	\begin{enumerate}
		\item The function $\varphi$ is differentiable $\sigma$-a.e., moreover $\partial^c\varphi(\eta)$ is single-valued whenever $\varphi$ is differentiable at $\eta$. The same holds for $\psi$.
		\item $\max(\varphi)+\min(\psi)=0$ and $\min(\varphi)+\max(\psi)=0$.
		\item If $\psi(\xi_0)=\min(\psi)$ then, for any $\xi\in\Sm$, $0\le \psi(\xi)-\psi(\xi_0)\le c(\xi_0,\xi)$.
		The same estimate holds for $\varphi$.
	\end{enumerate}
\end{proposition}
\begin{proof}
	According to Rademacher's theorem, the Lipschitz function $\varphi$ is differentiable almost everywhere. Moreover, assuming $\varphi$ is differentiable at $\eta\in\Sm$, it is well-known that $\varphi(\eta)+\psi(\xi)=c(\eta,\xi)$ admits a unique solution $\xi$ which can be expressed in terms of $\nabla\varphi(\eta)$ and the exponential map (see for instance \cite{McCann}).  
	
	Since the pair $(\varphi,\psi)$ is $c$-conjugate, Proposition \ref{prop-continuity-c-transform} implies  the functions $\varphi$ and $\psi$ are continuous on $\Sm$, in particular they have maximal and minimal points.
	
	For any $\zeta\in\Sm$,  $\varphi(\zeta)+\psi(\zeta)\le c(\zeta,\zeta)=0$ yields $\varphi(\zeta)+\min(\psi)\le 0$. Maximizing this with respect to $\zeta$ gives $\max(\varphi)+\min(\psi)\le 0$. Conversely, the cost function $c$ being non-negative we have, for any $\eta$ and $\xi$, $c(\eta,\xi)-\varphi(\eta)\ge-\max(\varphi)$. Taking the infimum over $\eta$ gives $\psi(\xi)\ge-\max(\varphi)$, therefore $\min(\psi)\ge-\max(\varphi)$. This proves the first equality in (2), the proof of the other is identical.
	
	Let $\xi_0\in\Sm$ be such that $\psi(\xi_0)=\min(\psi)$. For any $\eta\in\partial^c\psi(\xi_0)$,
	$$
	0 \le c(\eta,\xi_0) = \varphi(\eta) + \psi(\xi_0) \le \max(\varphi) + \min(\psi) = 0.
	$$
	This implies $\eta=\xi_0$ and $\varphi(\xi_0)=-\psi(\xi_0)$. Combining the latter equality with $\varphi(\xi_0)+\psi(\xi)\le c(\xi_0,\xi)$ yields
	$$
	0 \le \psi(\xi)-\psi(\xi_0)\le c(\xi_0,\xi),
	$$
	which proves (3) for $\psi$. The proof for $\varphi$ is similar.
\end{proof}

\begin{proposition}\label{prop-compacity}
	The set $\cB_M=\{ (\varphi,\psi)\ |\ \psi=\varphi^c,\ \varphi=\psi^c,\ \|\psi\|_\infty\le M \}$ is compact in $C^0$ topology.
\end{proposition}
\begin{proof}
	The elements of $\cB_M$ are $c$-conjugate pairs. Therefore, Proposition \ref{prop-continuity-c-transform} implies that any  $(\varphi,\psi) \in \cB_M$ is a pair of Lipschitz functions with Lipschitz constants only depending on upper bounds of $\varphi$ and $\psi$. Moreover, Proposition \ref{prop-min_c-concave} (2) implies $\|\varphi\|_\infty\le M$, so that $\cB_M$ is bounded in $C^{0,1}$ and thus compact in $C^0$.
\end{proof}

%
%
%

 %
\bibliographystyle{plain}

\bibliography{JB-PC-Curvature_hyperbolic_convex}

\begin{thebibliography}{10}

\bibitem{Alexandrov}
A.~D. Alexandroff.
\newblock Existence and uniqueness of a convex surface with a given integral
  curvature.
\newblock {\em C. R. (Doklady) Acad. Sci. URSS (N.S.)}, 35:131--134, 1942.

\bibitem{Alexandrov_book}
A.~D. Alexandrov.
\newblock {\em Convex polyhedra}.
\newblock Springer Monographs in Mathematics. Springer-Verlag, Berlin, 2005.
\newblock Translated from the 1950 Russian edition by N. S. Dairbekov, S. S.
  Kutateladze and A. B. Sossinsky, With comments and bibliography by V. A.
  Zalgaller and appendices by L. A. Shor and Yu. A. Volkov.

\bibitem{Bakelman}
I.~J. Bakelman.
\newblock {\em Convex analysis and nonlinear geometric elliptic equations}.
\newblock Springer-Verlag, Berlin, 1994.
\newblock With an obituary for the author by William Rundell, Edited by Steven
  D. Taliaferro.

\bibitem{Bayle}
V.~Bayle.
\newblock A differential inequality for the isoperimetric profile.
\newblock {\em Int. Math. Res. Not.}, 7:311--342, 2004.

\bibitem{Bertrand-1}
J.~Bertrand.
\newblock Prescription of {G}auss curvature on compact hyperbolic orbifolds.
\newblock {\em Discrete Contin. Dyn. Syst.}, 34(4):1269--1284, 2014.

\bibitem{Bertrand-2}
J.~Bertrand.
\newblock Prescription of {G}auss curvature using optimal mass transport.
\newblock {\em Geom. Dedicata}, 183:81--99, 2016.

\bibitem{Fillastre-Seppi}
F.~Fillastre and A.~Seppi.
\newblock Spherical, hyperbolic and other projective geometries: convexity,
  duality, transitions.
\newblock preprint, to appear Spherical and hyperbolic geometry revisited,
  2016.

\bibitem{Gerhardt}
C.~Gerhardt.
\newblock Minkowski type problems for convex hypersurfaces in hyperbolic space.
\newblock preprint.

\bibitem{Guan-Lin-Ma}
P.~Guan, C.~Lin, and X.-N. Ma.
\newblock The existence of convex body with prescribed curvature measures.
\newblock {\em Int. Math. Res. Not. IMRN}, 11:1947--1975, 2009.

\bibitem{Kohlmann}
P.~Kohlmann.
\newblock Curvature measures and {S}teiner formulae in space forms.
\newblock {\em Geom. Dedicata}, 40(2):191--211, 1991.

\bibitem{McCann}
R.~J. McCann.
\newblock Polar factorization of maps on {R}iemannian manifolds.
\newblock {\em Geom. Funct. Anal.}, 11(3):589--608, 2001.

\bibitem{Oliker-1}
V.~I. Oliker.
\newblock The {G}auss curvature and {M}inkowski problems in space forms.
\newblock In {\em Recent developments in geometry ({L}os {A}ngeles, {CA},
  1987)}, volume 101 of {\em Contemp. Math.}, pages 107--123. Amer. Math. Soc.,
  Providence, RI, 1989.

\bibitem{Oliker-2}
V.~I. Oliker.
\newblock Embedding {$\bold S^n$} into {$\bold R^{n+1}$} with given integral
  {G}auss curvature and optimal mass transport on {$\bold S^n$}.
\newblock {\em Adv. Math.}, 213(2):600--620, 2007.

\bibitem{ONeill}
B.~O'Neill.
\newblock {\em Semi-{R}iemannian geometry}, volume 103 of {\em Pure and Applied
  Mathematics}.
\newblock Academic Press, Inc. [Harcourt Brace Jovanovich, Publishers], New
  York, 1983.
\newblock With applications to relativity.

\bibitem{Santalo}
L.~A. Santal{\'o}.
\newblock {\em Integral geometry and geometric probability}.
\newblock Cambridge Mathematical Library. Cambridge University Press,
  Cambridge, second edition, 2004.
\newblock With a foreword by Mark Kac.

\bibitem{Schneider}
R.~Schneider.
\newblock {\em Convex bodies: the {B}runn-{M}inkowski theory}, volume 151 of
  {\em Encyclopedia of Mathematics and its Applications}.
\newblock Cambridge University Press, Cambridge, expanded edition, 2014.

\bibitem{Solanes}
G.~Solanes.
\newblock {\em Integral geometry and curvature integral in hyperbolic space}.
\newblock PhD thesis, Universitat Aut{\`o}noma de Barcelona, 2003.

\bibitem{Solanes-2}
G.~Solanes.
\newblock Integral geometry and the {G}auss-{B}onnet theorem in constant
  curvature spaces.
\newblock {\em Trans. Amer. Math. Soc.}, 358(3):1105--1115, 2006.

\bibitem{Solanes-Teufel}
G.~Solanes and E.~Teufel.
\newblock Integral geometry in constant curvature {L}orentz spaces.
\newblock {\em Manuscripta Math.}, 118(4):411--423, 2005.

\bibitem{Teufel}
E.~Teufel.
\newblock Differential topology and the computation of total absolute
  curvature.
\newblock {\em Math. Ann.}, 258(4):471--480, 1981/82.

\bibitem{Veronelli}
G.~Veronelli.
\newblock Boundary of convex sets in the hyperbolic space and curvature
  measures.
\newblock preprint, 2017.

\bibitem{Villani-1}
C.~Villani.
\newblock {\em Topics in optimal transportation}, volume~58 of {\em Graduate
  Studies in Mathematics}.
\newblock American Mathematical Society, Providence, RI, 2003.

\bibitem{Villani-2}
C.~Villani.
\newblock {\em Optimal transport}, volume 338 of {\em Grundlehren der
  Mathematischen Wissenschaften [Fundamental Principles of Mathematical
  Sciences]}.
\newblock Springer-Verlag, Berlin, 2009.
\newblock Old and new.

\end{thebibliography}

\vspace{10mm}

\begin{flushleft}
	J\'er\^ome Bertrand \\
	Institut de Math\'ematiques de Toulouse, CNRS, Univ. Paul Sabatier \\
	\texttt{jerome.bertrand\symbol{64}math.univ-toulouse.fr}
\end{flushleft}

\vspace{5mm}

\begin{flushleft}
	Philippe Castillon \\
	Institut Montpelli\'erain Alexander Grothendieck, CNRS, Univ. Montpellier \\
	\texttt{philippe.castillon\symbol{64}umontpellier.fr}
\end{flushleft}

\end{document}